\documentclass[11pt]{amsart} \textwidth=14.5cm \oddsidemargin=1cm
\usepackage{amsmath,amssymb,latexsym, amsthm, amscd, mathrsfs, stmaryrd}
\usepackage[linktocpage=true]{hyperref}
\usepackage{color}
\usepackage[all]{xy}

\usepackage{verbatim}

\newtheorem{thm}{Theorem} [section]
\theoremstyle{definition}

\newtheorem{example}[thm]{Example}
\newtheorem{rem}[thm]{Remark}
\theoremstyle{plain}
\newtheorem{prop}[thm]{Proposition}
\newtheorem{lem}[thm]{Lemma}
\newtheorem{cor}[thm]{Corollary}

\numberwithin{equation}{section}

\newcommand{\C}{\mathbb C}
\newcommand{\D}{D(2|1;\ka)}

\newcommand{\ep}{\epsilon}

\newcommand{\hf}{{\Small \frac12}}

\newcommand{\g}{\mathfrak{g}}
\newcommand{\gl}{\mathfrak{gl}}
\newcommand{\h}{\mathfrak{h}}
\newcommand{\ka}{\zeta}

\newcommand{\ft}{\mathfrak t}
\newcommand{\la}{\lambda}
\newcommand{\n}{\mathfrak n}
\newcommand{\mc}{\mathcal}
\newcommand{\mf}{\mathfrak}

\newcommand{\one}{{\ov 1}}

\newcommand{\ov}{\overline}
\newcommand{\OO}{\mathcal O}

\newcommand{\Q}{\mathbb Q}

\newcommand{\sll}{\mathfrak{sl}_2}

\newcommand{\Z}{\mathbb Z}
\newcommand{\oo}{{\ov 0}}
\newcommand{\hdel}{{h_{2\delta}}}
\newcommand{\hone}{h_{2\ep_1}}
\newcommand{\htwo}{h_{2\ep_2}}

\newcommand{\oa}{{\bar 0}}
\newcommand{\ob}{{\bar 1}}
\newcommand{\vare}{\epsilon} 

\newcommand{\KL}{\unlhd}

\newcommand{\Pl}{P_\la}
\newcommand{\Ml}{M_\la}
\newcommand{\Ll}{L_\la}
\newcommand{\Tl}{T_\la}

\newcommand{\Tw}{\mathbb T}

\newcommand{\ch}{{\rm ch}}

\newcommand{\LL}[1]{L_{{#1}}}
\newcommand{\M}[1]{M_{#1}}
\newcommand{\PP}[1]{P_{{#1}}}
\newcommand{\TT}[1]{T_{{#1}}}

\title[Blocks and characters of $D(2|1;\ka)$-modules of non-integral weights] {Blocks and characters of $D(2|1;\ka)$-modules of non-integral weights}

\author[Chen]{Chih-Whi Chen}
\address{Department of Mathematics, National Central University, Chung-Li, Taiwan 32054} \email{cwchen@math.ncu.edu.tw}
\author[Cheng]{Shun-Jen Cheng}
\address{Institute of Mathematics, Academia Sinica, Taipei, Taiwan 10617} \email{chengsj@math.sinica.edu.tw}
\author[Luo]{Li Luo}
\address{School of mathematical Sciences, Shanghai Key Laboratory of Pure Mathematics and Mathematical Practice, East China Normal University, Shanghai 200241, China}
\email{lluo@math.ecnu.edu.cn}

\allowdisplaybreaks

\begin{document}

\begin{abstract}
	
	We classify blocks in the BGG category $\mc O$  of modules of non-integral weights for the exceptional Lie superalgebra $\D$. We establish various reduction methods, which connect some types of non-integral blocks of $\D$ with integral blocks of general linear Lie superalgebras $\gl{(1|1)}$ and $\gl{(2|1)}$. We compute the characters for irreducible $\D$-modules of non-integral weights in $\mc O$.
\end{abstract}

\maketitle

\setcounter{tocdepth}{1}
\tableofcontents


	\section{Introduction}

\subsection{} Since a Killing-Cartan type classification of
finite-dimensional complex simple Lie superalgebras has been achieved by Kac \cite{Kac77} in 1977 there has been substantial efforts by mathematicians and physicists in the study of their representation theory.  The most important class of the simple Lie superalgebras is the so-called basic classical, which includes the classical series of types $ABCD$. Among these basic Lie superalgebras,  there are 3 exceptional ones: $\D$, $G(3)$ and $F(3|1)$.

\subsection{} The irreducible character problem is one of the main theme in representation theory of Lie (super)algebras.
While complete answers to the irreducible character problem in the BGG categories for  basic Lie superalgebras of types $ABCD$  are now known (see \cite{CLW11, CLW15,  Ba17, BLW17, BW18}), the BGG category of the exceptional Lie superalgebras were not until recently.

The study of character formulas in
the BGG category for exceptional Lie superalgebras was first initiated in  \cite{CW18, CW19}.   We recall that the Lie superalgebras $\D =\g_\oo \oplus\g_\one$ is a family of simple Lie superalgebras
of dimension $17$, depending on a  parameter $\ka \in\C\setminus\{0,-1\}$ with underlying even subalgebra $\g_\oo \cong \sll \oplus \sll \oplus \sll$. The BGG category for $\D$, coming from a triangular decomposition of $\D$, restricts to the BGG category for $\mf g_\oa$. In particular,  irreducible modules are therefore classified to their integral or non-integral highest weights with respect to the corresponding Borel subalgebra of $\mf g_\oa$. Character formulas for the tilting, or equivalently, irreducible $\D$-modules of integral highest weights in the BGG category have been established in \cite{CW19}.

\subsection{}The objective of the present paper is to address the irreducible character problem for the remaining cases not treated in \cite{CW19}, i.e., for the irreducible $\D$-modules of non-integral highest weights.

The blocks in the category $\mc O$ are divided into typical blocks and atypical blocks. It follows from works of Gorelik in \cite{Gor02} that the former are equivalent to blocks in the BGG category for $\mf g_\oa$. Our first result is a classification of atypical  blocks in the BGG category $\mc O$ of $\D$-modules of non-integral weights. We divide these blocks into three classes and call these blocks generic, $1$-integer, and $2$-integer blocks, according the number of integer values of the weights when evaluated at the even simple coroots. The order of the corresponding so-called integral Weyl group in each of these blocks increases with the number of such integer values, and hence the complexity of the linkage as well.

\subsection{}  In \cite{CMW13} Cheng, Mazorchuk, and Wang develop a reduction procedure which provides an equivalence of blocks in the BGG categories of the general linear Lie superalgebra. They employed twisting, odd reflection, and parabolic induction functors to reduce the irreducible character problem of an arbitrary weight to the problem of integer weight (also, see, e.g., \cite{CM16}). Consequently, the irreducible character problem in the BGG category has thus completely solved due to the validity of Brundan's super Kazhdan-Lusztig conjecture of integer weight $\gl(m|n)$-modules \cite{Br03}, which was established in \cite{CLW15} by Cheng, Lam, and Wang.

Along this line of reduction, we establish in the present paper equivalences of categories connecting the generic blocks and $1$-integer blocks for $\D$ with integral blocks for $\gl(1|1)$ and $\gl(2|1)$. The irreducible characters in the latter blocks are of course determined by Brundan's super Kazhdan-Lusztig theory (cf. \cite{CLW15}) and furthermore these blocks are known to be Koszul (cf. \cite{BLW17}).  In particular, we obtain complete solutions to the problems of irreducible characters, Loewy filtrations, Yoneda extension algebra of irreducible modules, etc., in these blocks for $\D$.

\subsection{} In the present paper, our strategy of constructing tilting modules in $2$-integer blocks cases   similar to the case in \cite{CW19}. To study these $2$-integer blocks in a uniform fashion we construct explicit Lie superalgebra isomorphisms between $\D,$  $D(2|1;\frac{1}{\ka})$,  and $D(2|1;-1-\ka)$. This then allows us to reduce all cases of $2$-integer blocks to one particular type. Subsequently, we adapt the main strategy in \cite[Subsection 1.3]{CW19} to solve the character problem for this particular type of $2$-integer blocks.

 Though one can solve the problem of finding tilting characters in non-integral blocks by applying suitable translation functor as in \cite{CW19}, the methods chosen in this paper by means of various equivalences of categories provide further additional information about these blocks and connection with classical theory.

\subsection{}  A motivation to study primitive spectrum, i.e. the description of all inclusions between annihilator ideals of the irreducible modules, originates from the open problem of classifying irreducible modules (see, e.g.,  \cite{Du77, BJ77, Jo79, Vo80, BB81, BK81, EW14}).

For basic classical Lie superalgebras, this problem has be extensively studied in literature, see, e.g., \cite{Vo80, Mu92, CM18, q2,  Letzter, CM16, Co16}. In particular, in \cite{Co16} there is an approach using the semi-simplicity of Jantzen middles. This allows to explicitly compute the inclusions of primitive ideals in terms of the Ext$^1$-quiver of irreducible modules in the BGG category. As a consequence of loc.~cit., the primitive spectrum problem for the general linear Lie superalgebras follows then from validity of Brundan's super Kazhdan-Lusztig conjecture.

It is natural to ask whether this same approach is viable for $\D$. As an application of our results,  we provide an affirmative answer to this primitive spectrum problem by reducing to the problem of finding ${\rm Ext}^1$-quiver of non-integral blocks. In particular, it is completely solved in the case of $1$-integer blocks.

\subsection{}The paper is organized as follows. In Section \ref{Sect::Preli}, we introduce the basic setup and provide some background material on the exceptional Lie superalgebra $\D$. In particular, we review the representation categories, the super Jantzen sum formula, isomorphisms of $\D$,  and irreducible characters of the general linear Lie superalgebras $\gl(1|1)$ and $\gl(2|1)$ that are needed in the sequel.

In Section \ref{Sect::blocks}, we classify the irreducible objects in any atypical block in $\mc O$. 
In Section \ref{sec:reduction}, we develop reduction methods through Arkhipov's twisting functor, parabolic induction functor and the various isomorphisms of $\D$ for different parameters $\zeta$.  In particular, this permits us to translate known closed formulas for irreducible character in integer blocks for $\gl(1|1)$ and $\gl(2|1)$ to closed formulas for irreducible character in generic and $1$-integer blocks for $\D$.

We establish in Section \ref{Sect::ChFormulae} the character formula of tilting modules in $2$-integer blocks. The main strategy of proof is to make use of translation functor, which is a similar strategy as the one used in \cite{CW19}.

Section \ref{Section::Prim} of the paper is devoted to reducing the problem of primitive spectra for $\D$ to the problem of finding ${\rm Ext}^1$-quiver. We then connect  the problem of primitive spectra for $1$-integer $\D$-blocks and $2$-integer $\D$-blocks to $\gl(1|1)$ and $\gl(2|1)$, respectively.

\vskip 0.5cm
{\bf Acknowledgment}. The first two authors are partially supported by MoST grants of the R.O.C.   The third author is partially supported by the Science and Technology Commission of Shanghai Municipality (grant No. 18dz2271000) and the NSF of China (grant No. 11871214). The first author thanks the Department of Mathematics of University of Virginia for hospitality and support during his visit in 2017.  The authors are grateful to Weiqiang Wang for interesting discussions.

\section{Preliminaries} \label{Sect::Preli}

Throughout the paper the symbols  $\mathbb C$, $\mathbb Z$, $\Z_{\geq 0}$, and  $\Z_{>0} $ stand for the sets of complex numbers, integers, non-negative and positive integers, respectively. Denote the abelian group of two elements by $\mathbb Z_2 =\{\oa,\ob\}$.  All vector spaces, algebras, tensor products, et
cetera, are over $\mathbb C$.

\subsection{The Lie superalgebras $\D$}
\subsubsection{}
Let $\ka\in\C\setminus\{0,-1\}$. The Lie superalgebra $\mathfrak{g}=\D$ is the contragredient Lie superalgebra associated with the following Cartan matrix:
\begin{align}\label{Cartan:matrix}
\begin{pmatrix}
0 & 1 & \ka \\
-1 & 2 & 0 \\
-1 & 0 & 2
\end{pmatrix}.
\end{align}
It is well known that $D(2|1,;\zeta)\cong D(2|1;\frac{1}{\zeta})\cong D(2|1;-1-\zeta)$. As we shall make use of these isomorphisms in the sequel, we shall construct them explicitly in Section \ref{isomo1}. We refer the reader to \cite[Section 2.5.2]{Kac77} for further details. 

Let $\{\hdel,\hone,\htwo\}$ be a basis for the Cartan subalgebra $\h\subset\mathfrak{g}$ and let $\{\delta, \ep_1,\ep_2\}$ be its dual basis. We equip $\h^*$ with a bilinear form $(\cdot,\cdot)$ such that $\{\delta, \ep_1,\ep_2\}$ are orthogonal and
\begin{align*}
(\delta, \delta) = -(1+\ka),
\quad
(\ep_1, \ep_1) = 1,
\quad
(\ep_2, \ep_2) = \ka.
\end{align*}

The simple coroots in the Cartan subalgebra $\h$ and the corresponding simple roots in $\h^*$ of $\D$ associated with the Cartan matrix \eqref{Cartan:matrix} are realized respectively as
\begin{align*}
\Pi^\vee&=\{\alpha_0^\vee=\frac{1+\ka}{2}\hdel+\hf\hone+\frac{\ka}{2}\htwo,\alpha_1^\vee=\hone,\alpha_2^\vee=\htwo\},\\
\Pi&=\{\alpha_0=\delta-\ep_1-\ep_2,\alpha_1=2\ep_1,\alpha_2=2\ep_2\}.
\end{align*}
The Dynkin diagram associated to $\Pi$ is depicted as follows:
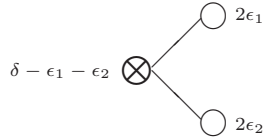
\begin{center}
	\setlength{\unitlength}{0.16in}
	\begin{picture}(4,6)
	\put(4,1.3){\makebox(0,0)[c]{$\bigcirc$}}
	\put(4,4.8){\makebox(0,0)[c]{$\bigcirc$}}
	\put(1.5,3){\makebox(0,0)[c]{$\bigotimes$}}
	\put(3.6,1.4){\line(-1,1){1.6}}
	\put(3.6,4.7){\line(-1,-1){1.6}}
	\put(5.2,4.8){\makebox(0,0)[c]{\tiny $2\ep_1$}}
	\put(5.2,1.2){\makebox(0,0)[c]{\tiny $2\ep_2$}}
	\put(-1,3){\makebox(0,0)[c]{\tiny $\delta-\ep_1-\ep_2$}}
	\end{picture}
\end{center}
 We let $e_{\alpha_0}$, $e_{\alpha_1}$, $e_{\alpha_2}$, $f_{\alpha_0}$, $f_{\alpha_1}$,and $f_{\alpha_2}$ be simple positive and negative root vectors so that together with the simple coroots they form a set of Chevalley generator for \eqref{Cartan:matrix}. To simplify notations we shall sometimes write $\{h_i,e_i,f_i\}_{i=0,1,2}$ for $\{\alpha_i^\vee, e_{\alpha_i}, f_{\alpha_i}\}_{i=0,1,2}$. We also let:
\begin{align}\label{D21basis}
\begin{split}
&e_{\delta+\ep_1-\ep_2}=[e_{0},e_{1}],\quad f_{\delta+\ep_1-\ep_2}=[f_{1},f_{0}],\\
&e_{\delta-\ep_1+\ep_2}=[e_{2},e_{0}],\quad f_{\delta-\ep_1+\ep_2}=[f_{0},f_{2}],\\
&e_{\delta+\ep_1+\ep_2}=[[e_{0},e_{1}],e_{2}],\quad f_{\delta+\ep_1+\ep_2}=[[f_{1},f_{0}],f_{2}].
\end{split}
\end{align}
We compute
\begin{align*}
e_{2\delta}=-\frac{1}{(1+\zeta)^2}[e_{\delta-\ep_1+\ep_2},e_{\delta+\ep_1-\ep_2}],\quad f_{2\delta}=[f_{\delta+\ep_1-\ep_2},f_{\delta-\ep_1+\ep_2}],\\
[e_{\delta\pm\ep_1\pm\ep_2},f_{\delta\pm\ep_1\pm\ep_2}]=\frac{1+\ka}{2}\hdel\mp\hf\hone\mp\frac{\ka}{2}\htwo, \quad [e_{2\delta},f_{2\delta}]=h_{2\delta}.
\end{align*}

Let $\Phi$, $\Phi_{\bar 0}$ and $\Phi_{\bar 1}$ stand for the sets of all, even and odd roots, respectively. Let $\Phi^+$ and $\Phi^-$ further denote the sets of positive and negative roots, respectively, with respect to $\Pi$ and set $\Phi^{\pm}_i=\Phi^{\pm}\cap\Phi_i$, for $i=\bar{0},\bar{1}$. Explicitly, we have
\begin{align*}
\Phi^+_{\bar 0}=\{2\delta,2\ep_1,2\ep_2\},\quad\Phi^+_{\bar 1}=\{\delta-\ep_1-\ep_2,\delta+\ep_1-\ep_2,\delta-\ep_1+\ep_2,\delta+\ep_1+\ep_2\}.
\end{align*}

 For $\alpha \in \Phi_\oo$, we denote by $\alpha^\vee \in \h$ the corresponding coroot so that
\[
\langle  \la,\alpha^\vee\rangle = 2(\la, \alpha)/(\alpha, \alpha),
\quad \forall \la\in \h^*.
\]

The reflection $s_\alpha$ on $\h^*$, for $\alpha \in \Phi_\oo$,  is defined as usual by letting $s_\alpha (\la) =\la - \langle  \la,\alpha^\vee\rangle \alpha$.

In the remainder of the article, we set $$(x,y,z):= x\delta+y\varepsilon_1+z\varepsilon_2\in \h^\ast.$$ The Weyl group  $W\cong\Z_2\times \Z_2\times \Z_2$ of $\D$ is the product of three copies of $\Z_2$.  Write $s_0,s_1,s_2$ for the simple reflections $s_{2\delta}, s_{2{\varepsilon_1}}$ and $s_{2{\varepsilon_2}}$, respectively, so that, for any $\la= (x,y,z) \in \h^\ast$, we have
\begin{equation}\label{def:s}
s_0(\la) = (-x,y,z),\quad  s_1(\la) = (x,-y,z),\quad s_2(\la) = (x,y,-z).
\end{equation}

A weight $\la=(x,y,z)$ is called {\em non-integral} if at least one among the complex numbers ${x,y,z}$ is not an integer.
For any $\la=(\la_0,\la_1,\la_2)\in \h^\ast$, we define the {\em associated integral Weyl group}
\begin{align}
&W_{\la}:=\langle s_i~|~\la_i\in\mathbb{Z} \rangle \subseteq W, \label{Eq::IntWeGr}
\end{align} which is generated by $s_i$ satisfying $\la_i\in\mathbb{Z}$.

A weight $\la=(x,y,z) \in \h^\ast$ is called {\em atypical} if $(\la ,\alpha) =0$ for some $\alpha=(1,\sigma,\tau) \in \Phi_{\bar 1}$ with $\sigma,\tau\in\{\pm 1\}$. Throughout the paper, for an atypical weight $\la \in \h^\ast$ with $(\la,\alpha)=0$ as above we let
\begin{equation}\label{eq:pd}
p_{\la, \alpha}:= -x+\sigma y,~d_{\la, \alpha}:= x-\tau z.
\end{equation}
If the context is clear, we simply denote $p_{\la, \alpha}$ and $d_{\la, \alpha}$ by $p$ and $d$, respectively. Note that $(\la ,\alpha) =0$ implies that
\begin{equation}\label{p=zetad}
p=\zeta d.
\end{equation}

We let $\rho$ denote the Weyl vector with corresponding to $\Pi$, that is,
\[
\rho=-\delta+\ep_1+\ep_2 \; (=-\alpha_0).
\]

\subsubsection{}\label{sec:oo'}
In the sequel we shall frequently need the following Dynkin diagram which is obtained from the standard diagram by applying the two odd reflections associated with the simple roots $\delta-\ep_1-\ep_2$ and $\delta+\ep_1-\ep_2$:
\begin{center}
	\setlength{\unitlength}{0.16in}
	\begin{picture}(4,6)
	\put(4,1.3){\makebox(0,0)[c]{$\bigcirc$}}
	\put(4,4.8){\makebox(0,0)[c]{$\bigcirc$}}
	\put(1.5,3){\makebox(0,0)[c]{$\bigotimes$}}
	\put(3.6,1.4){\line(-1,1){1.6}}
	\put(3.6,4.7){\line(-1,-1){1.6}}
	\put(5.2,4.8){\makebox(0,0)[c]{\tiny $2\ep_1$}}
	\put(5,1.2){\makebox(0,0)[c]{\tiny $2\delta$}}
	\put(-1,3){\makebox(0,0)[c]{\tiny $-\delta-\ep_1+\ep_2$}}
	\end{picture}
\end{center}
We denote the sets of simple roots by $\Pi'$, the associated Borel subalgebra by $\mf b'$ and the corresponding BGG category by $\OO'$.

\subsection{BGG category $\mc O$} Throughout the present article, we denote by $\OO$ the BGG category of $\D$-modules of non-integral weights
with respect to the triangular decomposition $\g=\n_-\oplus\h\oplus\n_+$, where $\n_{\pm}=\bigoplus_{\alpha\in\Phi^{\pm}} \mathfrak{g}_\alpha$, where $\mf g_\alpha$ denotes the root space of $\mf g$ corresponding to $\alpha$. That is, $\OO$ is the category of $\mf h$-semisimple, $\mf n_+$-locally finite, finitely generated $U(\D)$-modules. For $\la \in \h^*$, we define the Verma module by $M_\la= U(\g) \otimes_{U(\h+\n_+)} \C_{\la-\rho}$, where $\C_{\la-\rho}$ is the $1$-dimensional $U(\h+\n_+)$-module with $\h$ transforming by $\la-\rho$ and $\n_+$ acting trivially. The unique irreducible quotient module of $\M{\la}$ is  denoted by $\LL{\la}$. We let $P_\la$ denote the  projective cover of $L_\la$ in $\mc O$.

For any $M\in \mc O$, we let $\text{Rad} M, \text{Top} M $ and $\text{Soc} M$ denote the radical, the top and the  socle of $M$, respectively.

For $N\in \mc O$, we define its character $\text{ch}N$ by
$$\text{ch}N: =\sum_{\mu \in \h^\ast} \text{dim} N_\mu e^\mu,$$
where $e$ is an indeterminate and $N_\mu$ denotes the weight space of $N$ of weight $\mu$. Also, we let $[N:L_\mu]$ denotes the multiplicity of $L_\mu$ in $N$. If $N$ admits a Verma flag, then we let  $(N:M_\mu)$ denote the multiplicity of $M_\mu$ in $N.$

We denote by $T_\la$ the (indecomposable) tilting module of highest weight $\la-\rho$. A BGG type reciprocity for tilting modules is established in \cite{Soe98} and \cite{Br04} using the so-called {\em Soergel duality} functor $\mathbb S:\mc O\rightarrow \mc O$.
\begin{thm}\label{BGGS}\cite{Soe98} The functor $\mathbb S$ is an exact contragredient functor on the full subcategory of modules admitting Verma flags. Furthermore, we have $\mathbb S(P_\mu) = T_{-\mu}$ and $\mathbb S(M_{\mu}) = M_{-\mu}$, for any $\mu\in \h^\ast$. In particular, for any $\la, \mu \in \h^\ast$ we have
	\begin{equation}
	(\TT{-\la} : \M{-\mu})
	=(\PP{\la}: M_\mu)
	=[\M{\mu}: \LL{\la}].
	\end{equation}
\end{thm}


We recall the following proposition for the convenience of the reader.
\begin{prop}\cite[Proposition 2.2]{CW18} 	\label{prop:flags}
	Let $\la\in \h^\ast$, $\beta_i\in\Phi^+_{\bar 0}$, $1\le i\le k$, and $\beta,\gamma\in\Phi^+_{\bar 1}$. Let $w=s_{\beta_k}\cdots s_{\beta_2} s_{\beta_1}\in W$.
	\begin{itemize}
		\item[(1)]
		Suppose that $\langle\la,\beta_1^\vee\rangle\in \Z_{>0}$. Then $(T_\la: M_{s_{\beta_1} \la})>0$.
		
		\item[(2)]
		Suppose that
		$\langle s_{\beta_{i-1}}\cdots s_{\beta_{1}}\la,\beta_i^\vee\rangle \in  \Z_{>0}$,
		for all $i=1,\ldots,k$. Then $(T_\la:M_{w\la})>0$.
		
		\item[(3)]
		Suppose that $(\la,\beta)=0$. Then $(T_\la:M_{\la-\beta})>0$.
		
		\item[(4)]
		Suppose that $(\la,\beta)=0$ and $\langle s_{\beta_{i-1}}\cdots s_{\beta_{1}}(\la-\beta),\beta_i^\vee\rangle \in \Z_{>0}$
		for all $i=1,\ldots,k$.
		Then $(T_\la:M_{w(\la-\beta)})>0$.
		
		\item[(5)]
		Suppose that $(\la,\beta)=(\la-\beta,\gamma)=0$
		and  {$\beta\not\ge\gamma$}.  Then $(T_\la:M_{\la-\beta-\gamma})>0$.
		
		\item[(6)]
		Suppose that $(\la,\beta)=(\la-\beta,\gamma)=0$,
		 {$\beta\not\ge\gamma$}, $\langle s_{\beta_{i-1}}\cdots s_{\beta_{1}}(\la-\beta-\gamma),\beta_i^\vee\rangle > 0$, for all $i=1,\ldots,k$.
		Then $(T_\la:M_{w(\la-\beta-\gamma)})>0$.
	\end{itemize}
\end{prop}

We have the following.
\begin{cor}\label{coro:23}
Let $\la=(\la_0,\la_1,\la_2)\in \mathfrak{h}^*$ be a non-integral weight.
\begin{itemize}
\item[(1)] $(T_\la:M_{\la})=1$.
\item[(2)] If $(\la,\beta)=0$ for some $\beta\in\Phi^+_{\bar 1}$, then $(T_\la:M_{\la-\beta})>0$.
\item[(3)] If $\la_1=-1$ and $(\la,\beta)=0$ for some $\beta=(1,-1,\tau)\in\Phi^+_{\bar 1}$, then $(T_\la:M_{\la-(2,0,2\tau)})>0$. If $\la_2=-1$ and $(\la,\beta)=0$ for some $\beta=(1,\sigma,-1)\in\Phi^+_{\bar 1}$, then $(T_\la:M_{\la-(2,2\sigma,0)})>0$.
\item[(4)] Suppose $(T_\la:M_\mu)>0$ for some $\mu=(\mu_0,\mu_1,\mu_2)$. If $\mu_i\in\mathbb{Z}_{>0}$ for some $i\in\{0,1,2\}$, then $(T_\la:M_{s_i(\mu)})>0$.
\end{itemize}
\end{cor}

\begin{proof}
Parts (1), (2) and (4) are straightforward applications of Proposition \ref{prop:flags}.

For Part (3) we observe that if $\la_1=-1$ we have also $(\la-\beta,\gamma)=0$, for $\gamma=(1,1,\tau)$ and $\beta\not\ge\gamma$. Now a direct application of Proposition \ref{prop:flags}(5) gives $(T_\la:M_{\la-(2,0,2\tau)})>0$. In a similar fashion $(T_\la:M_{\la-(2,2\sigma,0)})>0$ is proved in the case $\la_2=-1$.
\end{proof}

A weight $\la \in \h^\ast$ is called {\it anti-dominant} if $\langle\la,\alpha^\vee\rangle\notin \Z_{>0}$, for all $\alpha \in \Phi^+_\oo$.

In \cite{Gor02} the notion of strongly typical weights was introduced. Moreover, it was shown that a strongly typical block of a basic Lie superalgebra is equivalent to a corresponding block of its even reductive subalgebra. Since for $D(2|1;\zeta)$ the notion of typical and strongly typical coincide, we have the following corollary.

\begin{cor} \label{Coro::TAtilt}
Let $\la \in \h^\ast$. Then $T_\la=M_\la$ if and only if $\la$ is typical and anti-dominant.
\end{cor}

\subsection{Isomorphisms $D(2|1;\zeta)$ with $D(2|1;\frac{1}{\zeta})$ and $D(2|1;-1-{\zeta})$}\label{isomo1}

\subsubsection{$D(2|1;\zeta)\cong D(2|1;\frac{1}{\zeta})$}\label{sec:D:iso1}

Let $$\{e_i,f_i,h_i~|~i=0,1,2\} \quad \mbox{and}\quad \{e'_i,f'_i,h'_i~|~i=0,1,2\}$$ be Chevalley generators of $D(2|1;\zeta)$ and $D(2|1;\frac{1}{\zeta})$, respectively. Define the following linear map $\phi$ from the space spanned by the Chevalley generators of $D(2|1;\frac{1}{\zeta})$ to $D(2|1;\zeta)$ as follows:
\begin{align*}
e'_0\stackrel{\phi}{\mapsto} \frac{1}{\zeta} e_0, \quad  & e'_1\stackrel{\phi}{\mapsto} e_2, \quad e'_2\stackrel{\phi}{\mapsto} e_1,\\
f'_0\stackrel{\phi}{\mapsto} f_0, \quad & f'_1\stackrel{\phi}{\mapsto} f_2,\quad f'_2\stackrel{\phi}{\mapsto} f_1,\\
h'_0\stackrel{\phi}{\mapsto}\frac{1}{\zeta} h_0, \quad & h'_1 \stackrel{\phi}{\mapsto}h_2, \quad h'_2\stackrel{\phi}{\mapsto} h_1.
\end{align*}
One easily verifies that the image of the Chevalley generators satisfies the relations of the Cartan matrix $$\begin{pmatrix}
                                                                                          0 & 1 & \frac{1}{\zeta} \\
                                                                                          -1 & 2 & 0 \\
                                                                                          -1 & 0 & 2
                                                                                        \end{pmatrix},$$
and hence $\phi$ induces a nonzero homomorphism from $D(2|1;\frac{1}{\zeta})$ to $D(2|1;\zeta)$, which must be an isomorphism by simplicity of these Lie superalgebras.

In terms of simple coroots we have in $D(2|1;\frac{1}{\zeta})$:
\begin{align*}
h'_{2\delta}=\frac{2}{1+\frac{1}{\zeta}}\left(h'_0-\frac{1}{2}h'_1-\frac{\frac{1}{\zeta}}{2}h'_2\right),
\end{align*}
and hence we have:
\begin{align*}
\phi(h'_{2\delta}) = \frac{2}{1+\frac{1}{\zeta}}\left(\frac{1}{\zeta}h_0-\frac{1}{2}h_2-\frac{\frac{1}{\zeta}}{2}h_1\right) =
\frac{2}{1+\zeta}\left(h_0-\frac{1}{2}h_1-\frac{\zeta}{2}h_2\right) = h_{2\delta}.
\end{align*}

Suppose that $M$ is a $D(2|1;\zeta)$-module. Then pulling back via the isomorphism $\phi$, we can regard $M$ as a $D(2|1;\frac{1}{\zeta})$-module. By the discussion above we see that if $(x,y,z)$ is a weight of $M$ as a $D(2|1;\zeta)$-module, then this weight translates to the weight $(x,z,y)$ of $M$ as a $D(2|1;\frac{1}{\zeta})$-module.

Since the map $\phi$ sends the standard Borel subalgebra of $D(2|1;\frac{1}{\zeta})$ to the standard Borel subalgebra of $D(2|1;\zeta)$, it follows that the pullback by $\phi$ defines an equivalence of highest weight categories between the respective BGG categories with parameters $\zeta$ and $\frac{1}{\zeta}$.
Furthermore, the character of a $D(2|1;\frac{1}{\zeta})$-module $M$ is obtained from the character of the same module $M$, but now regarded as a $D(2|1;{\zeta})$-module, by interchanging $\ep_1$ with $\ep_2$.

\subsubsection{$D(2|1;\zeta)\cong D(2|1;-1-\zeta)$}\label{sec:D:iso2}

Let $$\{e_i,f_i,h_i~|~i=0,1,2\} \quad \mbox{and}\quad \{e'_i,f'_i,h'_i~|~i=0,1,2\}$$ be Chevalley generators of $D(2|1;\zeta)$ and $D(2|1;-1-\zeta)$, respectively. Define the following linear map $\phi$ from the space spanned by the Chevalley generators of $D(2|1;-1-\zeta)$ to $D(2|1;\zeta)$ as follows:
\begin{align*}
e'_0\stackrel{\phi}{\mapsto}  f_{\delta+\ep_1-\ep_2}, \quad  & e'_1\stackrel{\phi}{\mapsto} e_1, \quad e'_2\stackrel{\phi}{\mapsto} e_{2\delta},\\
f'_0\stackrel{\phi}{\mapsto} -e_{\delta+\ep_1-\ep_2}, \quad & f'_1\stackrel{\phi}{\mapsto} f_1,\quad f'_2\stackrel{\phi}{\mapsto} f_{2\delta},\\
h'_0\stackrel{\phi}{\mapsto} -\frac{1+\ka}{2}\hdel+\hf h_1-\frac{\ka}{2}h_2, \quad & h'_1\stackrel{\phi}{\mapsto} h_1, \quad h'_2\stackrel{\phi}{\mapsto} h_{2\delta}.
\end{align*}
One easily verifies that the image of the Chevalley generators satisfies the relations of the Cartan matrix $$\begin{pmatrix}
                                                                                          0 & 1 & -1-{\zeta} \\
                                                                                          -1 & 2 & 0 \\
                                                                                          -1 & 0 & 2
                                                                                        \end{pmatrix},$$
and hence $\phi$ induces an isomorphism from $D(2|1;-1-{\zeta})$ to $D(2|1;\zeta)$.

In terms of simple coroots we have in $D(2|1;-1-{\zeta})$:
\begin{align*}
h'_{2\delta}=-\frac{2}{\zeta}\left(h'_0-\frac{1}{2}h'_1-\frac{-1-\zeta}{2}h'_2\right),
\end{align*}
and hence
\begin{align*}
\phi(h'_{2\delta}) = -\frac{2}{\zeta}\left(-\frac{1+\ka}{2}\hdel+\hf h_1-\frac{\ka}{2}h_2-\frac{1}{2}h_1-\frac{-1-{\zeta}}{2}h_{2\delta}\right) =
-\frac{2}{\zeta}\left(-\frac{\zeta}{2}h_2\right) = h_2.
\end{align*}

Suppose that $M$ is a $D(2|1;\zeta)$-module. Then pulling back via the isomorphism $\phi$, we can regard $M$ as a $D(2|1;-1-{\zeta})$-module. By the discussion above we see that if $(x,y,z)$ is a weight of $M$ as a $D(2|1;\zeta)$-module, then this weight translates to the weight $(z,y,x)$ of $M$ as a $D(2|1;-1-{\zeta})$-module.

Note that the map $\phi$ is not compatible with the standard Borel subalgebras. Thus, although the pullback by $\phi$ defines an equivalence of categories, it is not compatible with the highest weight structures.
Nevertheless, the character of a $D(2|1;-1-\zeta)$-module $M$ is obtained from the character of $M$, regarded as a $D(2|1;{\zeta})$-module, by interchanging $\delta$ with $\ep_2$.

\subsection{General linear Lie superalgebras} \label{Sect::Grogp} In this section we recall some basic facts about the BGG category  of the general linear Lie superalgebra $\gl(1|1)$ and $\gl(2|1)$. We shall write down closed formula for the irreducible characters in their principal block, as they will be used in the sequel to describe irreducible characters of $D(2|1;\zeta)$ in the generic and ``one-integer'' blocks.

Fix non-negative integers $m,n$. The general linear Lie superalgebra $\mathfrak{gl}(m|n)$ can be realized as $(m+n) \times (m+n)$ complex matrices
\begin{align} \label{glrealization}
\left( \begin{array}{cc} A & B\\
C & D\\
\end{array} \right),
\end{align}
where $A,B,C$ and $D$ are respectively $m\times m, m\times n, n\times m, n\times n$ matrices.   Let $E_{ab}$ be the elementary matrix in $\mathfrak{gl}(m|n)$ with $(a,b)$-entry $1$ and other entries $0$.
Let $\mf{h}_{m|n}$ and $\mf{h}_{m|n}^*$ be respectively the standard Cartan subalgebra of $\mf{gl}(m|n)$ and its dual,  with linear basis $\{ E_{ii} |  -m \leq i\leq n, i\not=0 \}$ and dual basis $\{ \delta^{\mf a}_i | -m \leq i\leq n, i\not=0\}$.

For a  triangular decomposition $\gl(m|n) = \mf n_-^{\mf a} \oplus  \mf h_{m|n} \oplus \mf n^{\mf a}$,  we let $\mc O(\gl(m|n), \mf h_{m|n} \oplus \mf n^{\mf a})$ denote the associated BGG category of $\mf h_{m|n}$-semisimple, $\mf n^{\mf a}$-locally finite, finitely generated $U(\gl(m|n))$-modules. Let $\mf b^s$ be the standard Borel subalgebra of $\gl(m|n)$ spanned by root vectors corresponding to the roots $\delta^{\mf a}_i-\delta^{\mf a}_j,$ for $-m\le i< j\le n$. Let $\rho_{\mf b}$ denote the Weyl vector corresponding to a Borel subalgebra $\mf b$. For $\mu \in \mf h_{m|n}^\ast$, we let $M^{\mf a}_{\mf b}(\mu)$   and $L^{\mf a}_{\mf b}(\mu)$ denote the $\mf b$-Verma module over $\mf{gl}(m|n)$ with $\mf b$-highest weight $\mu -\rho_{\mf b }$ and the irreducible quotient of $M^{\mf a}_{\mf b}(\mu)$, respectively. We shall at times drop $\mf b$ from the notation when the Borel is clear from the context and also use $(a_{-m},\ldots|b_1,\ldots)$ to denote the weight $a_{-m}\delta^{\mf a}_{-m}+\cdots+b_{1}\delta^{\mf a}_{1}+\cdots$.

In this paper we shall only need BGG categories of the Lie superalgebras $\gl(1|1)$ and $\gl(2|1)$.

The BGG category for $\gl(1|1)$ is the same as the category of finite-dimensional modules. The principal block is the block containing the trivial module and the irreducible modules are parameterized by $\{(a|-a)|a\in\Z\}$. The composition factors of the Verma modules are easily computed and we have the following formula regarded as identities in the Grothendieck group:
\begin{align*}
M^{\mf a}(a|-a)=L^{\mf a}(a|-a)+L^{\mf a}(a-1|a+1),\quad \forall a\in\Z.
\end{align*}
From this we obtain the inversion formula:
\begin{align}\label{gl11:char}
L^{\mf a}(a|-a)=\sum_{i=0}^\infty(-1)^iM^{\mf a}(a-i|-a+i),\quad \forall a\in\Z.
\end{align}
This category plays a central role in certain non-integral blocks of the queer Lie superalgebra $\mf{q}(2)$ (see \cite{Ch16}).

Now consider the Lie superalgebra $\gl{(2|1)}$.
Here, we let $\mf b$ denote the Borel subalgebra with positive roots $\delta^{\mf a}_{-2}-\delta^{\mf a}_{-1}, \delta^{\mf a}_{-2} -\delta^{\mf a}_1$ and $\delta^{\mf a}_{1} -\delta^{\mf a}_{-1}$. Note that $\mf b$ is obtained from $\mf b^s$ by applying odd reflection with respect to the root $\delta^{\mf a}_{-1}-\delta^{\mf a}_{1}$. For later use, we write explicit formulas for the coroots and the corresponding positive and negative root vectors:
\begin{align}\label{gl21coroots}
\begin{split}
&h_{\delta^{\mf a}_{-2}-\delta^{\mf a}_1}=E_{-2,-2}+E_{11},\quad h_{\delta^{\mf a}_1-\delta^{\mf a}_{-1}}=-E_{11}-E_{-1,-1},\quad h_{\delta^{\mf a}_{-2}-\delta^{\mf a}_{-1}}=E_{-2,-2}-E_{-1,-1},\\
&e_{\delta^{\mf a}_{-2}-\delta^{\mf a}_1}=E_{-2,1},\quad e_{\delta^{\mf a}_1-\delta^{\mf a}_{-1}}=E_{1,-1},\quad e_{\delta^{\mf a}_{-2}-\delta^{\mf a}_{-1}}=E_{-2,-1},\\
&f_{\delta^{\mf a}_{-2}-\delta^{\mf a}_1}=E_{1,-2},\quad f_{\delta^{\mf a}_1-\delta^{\mf a}_{-1}}=-E_{-1,1},\quad f_{\delta^{\mf a}_{-2}-\delta^{\mf a}_{-1}}=E_{-1,-2}
\end{split}
\end{align}

Since $\mf b$ is obtained from $\mf b^s$ by applying odd reflection with respect to the root $\delta^{\mf a}_{-1}-\delta^{\mf a}_1$, we have the following lemma, by  {\cite[Section \S9.4]{CW08} and \cite[Lemmas 6.1,  6.2]{CLW15}}, for characters of the Verma modules  {(cf. \cite{Ar19} for a second approach)}.

\begin{lem} \label{Lem::chvermagl21b}
	We have the following composition factors for Verma modules in the principal block of $\gl(2|1)$ with respect to the Borel subalgebra $\mf b$:
	\begin{align*}
	&M^{\mf a}(0,-i|i) = L^{\mf a}(-i,0|i)+L^{\mf a}(0,-i|i)+L^{\mf a}(0,-i+1|i-1)+ L^{\mf a}(0,-i-1|i+1),\quad (i\ge 1),\\
	&M^{\mf a}(-i,0|i) = L^{\mf a}(-i,0|i)+L^{\mf a}(-i-1,0|i+1),\quad (i\ge 1),\\
	&M^{\mf a}(0,0|0) = L^{\mf a}(0,0|0)+L^{\mf a}(-1,0|1)+L^{\mf a}(0,1|-1),\\
	&M^{\mf a}(0,1|-1) = L^{\mf a}(0,1|-1)+L^{\mf a}(0,2|-2),\\
	&M^{\mf a}(1,0|-1) = L^{\mf a}(1,0|-1)+L^{\mf a}(0,1|-1) + L^{\mf a}(0,2|-2) + L^{\mf a}(0,0|0), \\
	&M^{\mf a}(0,i|-i) = L^{\mf a}(0,i|-i)+L^{\mf a}(0,i+1|-i-1),\quad (i\ge 2),\\
	&M^{\mf a}(i,0|-i) = L^{\mf a}(i,0|-i)+L^{\mf a}(0,i|-i)+L^{\mf a}(i-1,0|-i+1))+L^{\mf a}(0,i+1|-i-1),\quad (i\ge 2).
	\end{align*}
\end{lem}

Inverting the decomposition matrix we get the following closed formula.
\begin{thm}\label{char:gl21}
The irreducible modules in the principal block of $\OO$ with respect to $\mf{b}$ in terms of Verma modules are given as follows:
\begin{align*}
&L^{\mf a}(n,0|-n)=\sum_{i=-\infty}^n(-1)^{n+i} M(i,0|,-i)+\sum_{j=-\infty}^{-n-1}(-1)^{n+j}M(0,-j|j), \quad n\ge 0,\\
&L^{\mf a}(0,-n|n)=\sum_{i=-\infty}^n(-1)^{n+i} M(0,-i|,i)+\sum_{j=-\infty}^{-n-1}(-1)^{n+j}M(j,0|-j), \quad n\ge 0,\\
&L^{\mf a}(-n,0|n)=\sum_{i=-\infty}^{-n}(-1)^{n+i} M(i,0|,-i), \quad n> 0,\\
&L^{\mf a}(0,n|-n)=\sum_{i=-\infty}^{-n}(-1)^{n+i} M(0,-i|,i), \quad n> 0.
\end{align*}
\end{thm}

\section{Classification of blocks} \label{Sect::blocks}

We denote the (indecomposable) block in $\OO$ containing $L_{\la}$ by $\OO_{\la}$.   Also, we let $\mathrm{Irr}\OO_{\la}$ denote the set of highest weights of irreducible objects in $\OO_{\la}$.

A weight $\la =(x,y,z)\in \h^\ast$ is called  {\em generic} if $x,y,z \notin \Z$. Furthermore, $\la$ is called an $1$-integer weight (respectively, $2$-integer weight) if the number of integers among $\{x,y,z\}$ equals $1$ (respectively, $2$). Also, accordingly, the corresponding block $\mc O_\la$ is referred to as generic, $1$-integer or $2$-integer.

We classify in this section  $\mathrm{Irr}\OO_{\la}$ for atypical and non-integral weights $\la$.


\subsection{Generic blocks} {We first  describe $\mathrm{Irr}\OO_{\la}$ for arbitrary atypical and generic weights $\la$.}

\begin{thm} \label{thm::genericthm}
Let $\la $ be an atypical and generic weight  with   $(\la, \alpha) =0$, for some $\alpha \in \Phi_\ob^+$. Then $$ T_{\la+\alpha} = P_\la, ~\ch P_\la =\ch M_\la +\ch M_{\la + \alpha},~ \ch L_{\la}= \frac{ \ch M_{\la}  }{1+e^{-\alpha}},$$ and so $\mathrm{Irr}\OO_{\la} = \la+ \mathbb Z \alpha$.
\end{thm}
\begin{proof}
For any $k\in \Z$, we  consider the character of the projective cover $P_{\la +k \alpha}$.  By Jantzen filtration for basic Lie superalgebras, see, e.g., \cite[(10.3)]{Mu12}, we have a filtration $$M_{\la+k\alpha} \supset M_{\la+k\alpha} ^1 \supset M_{\la+k\alpha}^2 \supset \cdots \supset M_{\la+k\alpha}^\ell,$$ such that $\sum_{i>0}\text{ch}M^i_{\la+k\alpha} = \frac{ \ch M_{\la+(k-1)\alpha}}{1+e^{-\alpha}}$ with $  M^1_{\la+k\alpha} = \text{Rad} M_{\la+k\alpha}$. In particular, we observe that $\ch M_{\la+k\alpha} $ can be expressed as a sum of two characters of modules  $$\ch M_{\la+k\alpha} =\frac{ \ch M_{\la+k\alpha}+\ch M_{\la+(k-1)\alpha}}{1+e^{-\alpha}} = \sum_{i>0}\text{ch}M^i_{\la+(k+1)\alpha}+\sum_{i>0}\text{ch}M^i_{\la+k\alpha}.$$
	This means that $\frac{ \ch M_{\la+k\alpha}  }{1+e^{-\alpha}}=\sum_{i>0}\text{ch}M^i_{\la+(k+1)\alpha} =\ch L_{\la +k\alpha}$ since $  M^1_{\la+k\alpha} = \text{Rad} M_{\la+k\alpha}$.  Consequently, we arrive at $\ch M_{\la +k\alpha} = \ch L_{\la+k\alpha} +\ch L_{\la +(k-1)\alpha}$, and therefore we conclude that $P_{\la +k\alpha} = M_{\la +k\alpha} + M_{\la +(k+1)\alpha}. $
	
	{ By \cite[Theorem 4.4 (6)]{CCC19}, $P_\la$ is a tilting module in $\mc O_\la$. Consequently, we have $P_\la =T_{\la+\alpha}.$ This completes the proof.}
\end{proof}

\subsection{Casimir element} For a given $\la \in \h^\ast$, the eigenvalue of Casimir element acting on $L_\la$ is denoted by $c_\la$. Let $\la = (x,y,z)\in\h^\ast$, we recall the following formula (cf. \cite{CW19}):
\begin{align}
&c_\la = -(1+\zeta)x^2 + y^2 +\zeta z^2. \label{Eq::Car}
\end{align} For given $ \mu \in \h^\ast$, we will freely use the fact that $L_\mu\in \mc O_\la$ implies that $c_\la =c_\mu$.

Recall $p$ and $d$ defined in \eqref{eq:pd}.
The following fundamental lemma calculates the value $c_\la$.
\begin{lem} \label{lem::CasValue}
	If $\la $ is atypical then $c_\la = p(p+d)$.
\end{lem}
\begin{proof}
	
	By \eqref{eq:pd} we can write $(x,y,z)=(x,\sigma(x+p),\tau(x-d))$ so that
	\begin{align}\label{aux:01}
	c_\la=-(1+\zeta)x^2+(x+p)^2+\zeta(x-d)^2.
	\end{align}
	If $d\not=0$, then we replace $\zeta$ in \eqref{aux:01} by $p/d$ and get $c_\la=p^2+pd$.
	
	If $d=0$, then by \eqref{p=zetad} we have $p=0$ and hence by \eqref{aux:01} $c_\la=0$.
	
\end{proof}

{ The following corollary is a direct consequence of Lemma \ref{lem::CasValue}. \begin{cor} \label{Coro::czero}
		Let $\la$ is atypical and non-integral. Then $c_\la =0$ if and only if $\la$ is generic.
\end{cor}}

\subsection{Blocks}
\label{Sec::IrrBlocks}
In this section,  we set $\la=(x,y,z) \in \h^\ast$ to be an  atypical non-integral weight with $(\la, \alpha) =0 $ for $\alpha= (1,\sigma, \tau) \in \Phi_\ob^+$.
We recall that $W_\la$ denotes the integral Weyl group of $\la$. The following result is our first main theorem in the present article.
\begin{thm} \label{Prop::IrrBlocks}  	 For arbitrary parameter $\zeta$,
	we have  $$\mathrm{Irr}\OO_{\la} = W_\la (\la+  \mathbb{Z}(1,\sigma ,\tau)).$$
	
	
	
	
	
	
	
\end{thm}

Before proving Theorem \eqref{Prop::IrrBlocks}, we first develop some useful tools for classifying  non-integral blocks. Recall $p$ and $d$ given  in \eqref{eq:pd}.
\begin{lem} \label{lem::blocksmainlem} Let $\la = (x,y,z)$ be an atypical non-integral weight with $c_\la \neq 0$.  Suppose that $L_{\la +(a,b,c)} \in \OO$, for some $a,b,c \in \Z$, with $c_{\la+(a,b,c)}=c_\la$.
	
	(1) We have \begin{align}& a \equiv   b \equiv   c ~ (\text{mod}~2). \label{Eq::Rule1Eq}\end{align}
	
	(2) There exists $\sigma',\tau' \in \{\pm 1\}$ such that \begin{align}&\zeta = \frac{-x+\sigma y}{x-\tau z} = \frac{-(x+a)+\sigma' (y+b)}{(x+a)-\tau' (z+c)}.\end{align} 
	
	(3) There is $\eta \in \{\pm 1\}$ such that
	\begin{align}\label{xyEq}&(\eta -1)x+(\sigma -\sigma'\eta)y +\eta(a-\sigma'b) =0,\\
	\label{xzEq}&(1-\eta)x+(\eta\tau'-\tau)z+\eta(-a+\tau' c) =0.
	\end{align}
	
	(4) Assume that $\eta =1$. Then we have
	\begin{align}
	&y\notin \Z \Rightarrow \sigma =\sigma',~ a=\sigma b.  \label{Eq::eta1y}\\
	&z\notin \Z \Rightarrow \tau =\tau',~a=\tau c. \label{Eq::eta1z}
	\end{align}

	(5) Assume that $\eta =-1$. We have the following.
	\begin{enumerate}
		\item[(i)] If $x\notin \Z$, then
   \begin{align}
	&\sigma =\sigma',~\tau =\tau',~p=\frac{a-\sigma b}{2},~\text{and}~d=\frac{-a+\tau c}{2}. \label{Eq::eta-1x}
	\end{align} In particular, in this case we have $\zeta \in \Q$.
	
\item[(ii)] If $x\in \Z$, then
	\begin{align}
&y\notin \Z \Rightarrow \sigma =-\sigma',~x=\frac{-a-\sigma b}{2}.  \label{Eq::eta-1y}\\
&z\notin \Z \Rightarrow \tau =-\tau',~x=\frac{-a-\tau c}{2}. \label{Eq::eta-1z}
\end{align}
	\end{enumerate}
	
	(6)  Assume that $\eta =-1$.  If furthermore $\sigma =\sigma'$, $\tau =\tau'$ and $a-\tau c\neq 0$ then \begin{align}
	&\zeta =  \frac{-a +\sigma b}{a-\tau c}\in \Q. \label{Eq::eta-1zetavalueQ}
	\end{align}
\end{lem}
\begin{proof}
	Since $(a,b,c)$ lies in the root lattice, Part (1) follows.
	
 By assumption $\la +(a,b,c)$ is atypical with $c_{\la+(a,b,c)}=c_\la\not=0$. This proves Part (2).

 We now prove   Part $(3)$. Observe that by Lemma \ref{lem::CasValue} we have the following identities
	\begin{align}&(-x+\sigma y)^2(1+\zeta^{-1}) = c_\la = c_{\la+(a,b,c)} = (-(x+a)+\sigma' (y+b))^2(1+\zeta^{-1}),\\
	&(x-\tau z)^2(\zeta+\zeta^{2}) = c_\la = c_{\la+(a,b,c)} = ((x+a)-\tau' (z+c))^2(\zeta+\zeta^{2}),
	\end{align}
	which in turn imply the identities
	\begin{align}& p^2 =(p+(\sigma'-\sigma)y +\sigma' b -a)^2,\label{eq:aux001}\\
	&d^2=(d+(\tau-\tau')z +a-\tau'c)^2.\label{eq:aux002}
	\end{align}
Now, if $(\sigma'-\sigma)y +\sigma' b -a \neq 0$, then \eqref{eq:aux001} gives $2p+(\sigma'-\sigma)y +\sigma' b -a=0$, which is equivalent to \eqref{xyEq} with $\eta=-1$. On the other hand, $(\sigma'-\sigma)y +\sigma' b -a = 0$ is equivalent to \eqref{xyEq} with $\eta=1$.  Similarly, $(\tau -\tau')z +a -\tau' c \neq 0$ together with \eqref{eq:aux002} implies \eqref{xzEq}  with $\eta=-1$, while $(\tau -\tau')z +a -\tau' c = 0$ is equivalent to \eqref{xzEq}  with $\eta=1$. This completes the proof of  Part (3).

	We now prove  Part (4). If  $\sigma =-\sigma'$, then  by \eqref{xyEq} we would have
	\begin{align}
	&y=\frac{- b - \sigma a}{2},
	\end{align}
and so by \eqref{Eq::Rule1Eq} $y$ would be an integer.
	Similarly, if $\tau =-\tau'$, then  by \eqref{xzEq} we would have
	\begin{align}
	&~z=\frac{-\tau a -c}{2}.
	\end{align}

	We now prove Part  (5).  To prove  Part  (5)-(i), we suppose on the contrary that
	\begin{align}
	&\sigma =-\sigma' \Rightarrow x =\frac{-\sigma b -a}{2},\\
	&\tau =-\tau' \Rightarrow x = \frac{-a -\tau c}{2},
	\end{align}  by \eqref{xyEq} and \eqref{xzEq}. Either of them contradicts to \eqref{Eq::Rule1Eq} and the fact that $x\notin \Z$.  Part (5)(ii) is proved similarly.
	
	We now prove  Part  $(6)$. In this case we have
	\begin{align}
	&\zeta = \frac{p}{d} =\frac{p-a+\sigma b}{d+a-\tau c}.
	\end{align} Since $a-\tau c\neq 0$, the proof is completed.
\end{proof}


We are now in a position to give a proof of Theorem \ref{Prop::IrrBlocks}. Recall that $\mc O_\la$ is called generic if $x,y,z\notin \Z$.

\begin{proof}[Proof of Theorem \ref{Prop::IrrBlocks}]
	
 	We first note that \cite[Proposition 2.1, 2.2]{CW18} implies that
 	\begin{align} \label{Eq::IrrInclusion}
 	&\mathrm{Irr}\OO_{\la} \supset  W_\la    (\la+  \mathbb{Z}\cdot  (1,\sigma ,\tau)).
 	\end{align}
	It remains to show that $\text{Irr}\OO_{\la} \subset  W_\la    (\la+  \mathbb{Z}\cdot  (1,\sigma ,\tau))$. { By Corollary  \ref{Coro::czero}, $c_\la=0$ implies that $\la$ is generic and so the proof follows from Theorem \ref{thm::genericthm}. Therefore, we assume that $c_\la \neq 0$ in the present proof.}

\vskip 0.5cm

 Let $L_{\la +(a,b,c)} \in \OO_\la$ and $\mu := \la +(a,b,c)$, we will proceed with direct computation using results and notations in the proof of Lemma \ref{lem::blocksmainlem}. We will freely use \eqref{xyEq} and \eqref{xzEq} in the following calculations.  \vskip 0.5cm

{\bf Case (I)  $x,y,z \notin \mathbb Z$}:  This case follows from Theorem \ref{thm::genericthm}.

\vskip 0.5cm
We note that if one of $x,y,z$ is an integer, then without loss of generality we may assume that this number is zero  by \eqref{Eq::IrrInclusion}.  Consequently, the remaining cases are listed as follows.

\vskip 0.5cm
{\bf Case (II)  $x=0 ~\text{and} ~y,z \notin \mathbb Z$}: \\
 {If $\eta=1$} we have by \eqref{Eq::eta1y} and \eqref{Eq::eta1z} that $a= \sigma b =  \tau c$. It follows that  $\mu =\la +a(1,\sigma,\tau)$.

  If $\eta  =-1$ then $a= -\sigma b = -\tau c$ by \eqref{Eq::eta-1y} and \eqref{Eq::eta-1z}. Namely, we have {$$\mu = \la-(-a,\sigma a,\tau a).$$}  This completes the proof of Case (II).

\vskip 0.5cm
{\bf Case (III)   $y=0 ~\text{and} ~x,z \notin \mathbb Z$}: \\
Observe that $\eta  =1$. To see this, suppose on the contrary that $\eta =-1$ then \eqref{xyEq} implies that $-2x =a-\sigma' b$, which contradicts to \eqref{Eq::Rule1Eq}  since $x\notin \Z$.  Now by \eqref{xyEq} and \eqref{Eq::eta1z} we have $a =\sigma' b \in \{\pm b\}$ and $\tau = \tau'$, $  c=\tau a$.  This means that
$$\mu=\la+(a,\pm a,\tau a).$$  This completes the proof of Case (III).

\vskip 0.5cm
{\bf Case (IV)    $z=0 ~\text{and} ~x,y \notin \mathbb Z$}: \\
Observe that $\eta  =1$. To see this, suppose on the contrary that $\eta =-1$ then \eqref{xzEq} implies that $2x =-a+\tau' c$, which contradicts to \eqref{Eq::Rule1Eq} since $x\notin \Z$.  Now by \eqref{xzEq} and \eqref{Eq::eta1y} we have  and $\sigma = \sigma'$, $b= \sigma a$ and $a =\tau' c \in \{\pm c\}$.  This means that
$$\mu=\la+(a,\sigma a,\pm a),$$ as desired.

\vskip 0.5cm
{\bf Case (V)   $x=0,y\in \Z ~\text{and}~ z \notin \mathbb Z$}: \\
 If $\eta  =1$ then by \eqref{Eq::eta1z} we have $a = \tau c$.
 By  \eqref{xyEq}, we have the following two possible situations:
 \begin{align}
 &\sigma =\sigma' \Rightarrow a = \sigma b. \label{Eq::24} \\
 &\sigma =-\sigma' \Rightarrow b=-2y-\sigma a. \label{Eq::25}
 \end{align}    If \eqref{Eq::24} holds then $$\mu=\la+(a,\sigma a, \tau a)=(a,y+\sigma a,z+\tau a).$$
 If \eqref{Eq::25} holds then  $y+b = -y -\sigma a$ and so $$ {\mu =(a, -y-\sigma a,z+\tau a)}.$$ This completes the proof of this subcase.

We now turn to the case $\eta =-1$. By \eqref{Eq::eta-1z} we have  $a = -\tau c$. By \eqref{xyEq}, we have the following two possible cases
\begin{align}
&\sigma =-\sigma' \Rightarrow a = -\sigma b.  \label{Eq::26} \\
&\sigma =\sigma' \Rightarrow  b= \sigma a - 2y \label{Eq::27}
\end{align}  If \eqref{Eq::26} holds and $a':=-a$ then we have
$$\mu = \la +( {-a'}, \sigma a', \tau a')=(-a',y+\sigma a',z+\tau a').$$
 If \eqref{Eq::27} holds  and $a':=-a$  then  $y+b = -y +\sigma a =-(y-\sigma a)$ and so  $$\mu =( {-a'}, -(y+\sigma a'),  {z}+\tau a').$$
 {Both such expressions lie in $W_\la(\la+\Z (1,\sigma,\tau))$.} This completes the proof of Case (V).

\vskip 0.5cm
{\bf Case (VI)    $y=0, z\in \Z ~\text{and} ~x \notin \mathbb Z$}: \\
We firstly note that $\eta \neq -1$ by \eqref{xyEq} and the facts $y=0$, $x\notin \Z$. Now we have  $\eta =1$.  By \eqref{xyEq}, we have $a=\sigma' b$.

 By \eqref{xzEq}, we have one of the following possible situations:
\begin{align}
&\tau = \tau' ~\text{and}~ a = \tau c. \label{Eq::31}  \\
&\tau = -\tau' ~\text{and}~ c = -2z-\tau a. \label{Eq::32}
\end{align}

 If \eqref{Eq::31} holds, then $b \in \{\pm a\}$, $a=\tau c$.
 Therefore we have $\mu = (x+a, \pm \sigma a,z+\tau a)$.

 If \eqref{Eq::32} holds, then   $b\in \{\pm a\}$, $z+c =-z -\tau a$.  Therefore we have $$\mu = (x+a, \pm \sigma a,-(z+\tau a)).$$

\vskip 0.5cm
{\bf Case (VII)   $x=0, z\in \Z ~\text{and} ~y \notin \mathbb Z$}: \\

We firstly consider $\eta = 1$. By \eqref{Eq::eta1y} it follows that $a=\sigma b$. Also, by \eqref{xzEq}, we have the following two possible cases:
\begin{align}
&\tau =\tau' \Rightarrow a =\tau c. \label{Eq:33} \\
&\tau = -\tau' \Rightarrow c=-\tau a -2z.  \label{Eq::34}
\end{align}
If \eqref{Eq:33} holds, then $\mu =\la +a(1,\sigma,\tau)$. \\
If \eqref{Eq::34} holds, then $c+z = -\tau a-z$ and so we have $$\mu = (a, y+\sigma a, -(z+\tau a)).$$

We now assume that $\eta =-1$. By \eqref{Eq::eta-1y} we have $a=-\sigma b$. Again, by \eqref{xzEq} we have the following two possible situations:
\begin{align}
&\tau =-\tau' \Rightarrow a= -\tau c. \label{Eq::35}\\
&\tau =\tau' \Rightarrow c=\tau a-2  z.\label{Eq::36}
\end{align}
If \eqref{Eq::35} holds   and $a' := -a$, then $\mu =\la +(-a',\sigma a',\tau a')=(-a',y+\sigma a',z+\tau a')$.\\
If \eqref{Eq::36} holds and $a' := -a$, then we have $$\mu = ( {-a'}, y+ \sigma a', -(z+\tau a')).$$  {Again, both such expressions lie in $W_\la(\la+\Z (1,\sigma,\tau))$.} This completes the proof.
\end{proof}




\section{Reduction methods and characters in the generic and $1$-integer cases} \label{sec:reduction}

\subsection{Arkhipov twisting functors} \label{subSect::Twisting}
We recall the construction of Arkhipov's {\em twisting functor} of~\cite{Ar97}.   Fix a simple even root $\alpha \in \Phi_\oo^+$ and a non-zero root vector $X \in (\mathfrak{g}_{\bar{0}})_{-\alpha}$.
Then the twisting functor $\Tw_{s_{\alpha}}$ associated to $\alpha$ is the functor obtained by tensoring on the left with the bimodule $^{\varphi}(U'_{\alpha}/U)$:
$$
\Tw_{s_{\alpha}}(-) :=  ^{\varphi}(U'_{\alpha}/U)\otimes_U - : \mathcal{O} \rightarrow \mathcal{O},
$$
where $U'_{\alpha}$ is the the Ore localization of $U$ with respect to powers of $X$. Here $^{\varphi}(U'_{\alpha}/U)$ denotes $U'_{\alpha}/U$ where the $\g$-action is twisted by an automorphism
$\varphi$ of $\mathfrak{g}$ that maps $(\mathfrak{g}_{i})_{\beta}$ to  $(\mathfrak{g}_{i})_{s_{\alpha}(\beta)}$ for all simple roots $\beta$ and $i\in \{\bar{0}, \bar{1}\}$.  We refer the reader to  \cite{CMW13} (also, see, e.g.,  \cite{AS03} and \cite[Section~5]{CM16}) for more details.


We denote by  $\mc O[\chi]$  the subcategory of $\mc O$ corresponding to the central character $\chi=\chi_\la: \mc Z(\mf g) \rightarrow \C$.  It is proved  in \cite[Proposition~5.11]{CM16} that the left derived functor $\mc L \Tw_{s_\alpha}$ provides an auto-equivalence of the bounded derived category   $ D^b(\mc O[\chi])$. Assume that $\chi_\la$ is the central character given by a weight $\la \in \h^\ast$ with  $\langle\la,\alpha^\vee\rangle\notin \Z$ then  $\mc L\Tw_{s_\alpha}$ restricts to an  auto-equivalence $\Tw_{s_\alpha}$ on $\mc O{[\chi_\la]}$ (cf. \cite[Theorem 2.1]{CMW13}, \cite[Proposition 5.11]{CM16}). 
The highest weight of a simple  module $\mathbb T_{s_\alpha}L_\mu$ is not  controlled by the usual action of $W$ (see, e.g., \cite[Example 3.3]{Ch16} for an example coming from the queer Lie superalgebra). 
Instead, there is a star action $\ast$ of an infinite Coxeter group $\tilde{W}$ associated to $W$ defined in \cite[Section 8]{CM16} for basic classical Lie superalgebras, which was first introduced in \cite{GG13} by Gorelik and Grantcharov. The correct description is given by the   action $\ast$, that is, $\Tw_{s_\alpha}L_\mu = L_{s_{\alpha}\ast (\mu-\rho)+\rho}$ by \cite[Lemma 8.3]{CM16}. We may conclude from \cite{CM16} the following lemma which gives rise to equivalences of (indecomposable) blocks using the twisting functors.
\begin{lem}  \label{lem::equiTs}
  If $\langle\la,\alpha^\vee\rangle\notin\Z$ then
	$\Tw_{s_\alpha}:\mc O_\la \rightarrow \mc O_{s_\alpha\la}$ is an equivalence sending $L_\mu$ to $L_{s_{\alpha}\ast (\mu-\rho)+\rho}$.
\end{lem}
\begin{proof}  
  By  \cite[Lemma 5.5]{CM16} we have    $\text{ch}\Tw_{s }  M_{\mu}  =\text{ch}M_{{s } \mu}$, for any $\mu  \in \h^\ast$ and $s\in W$. Therefore $[M_{s_\alpha\mu}:\Tw_{s_\alpha}  L_{\mu}]>0$, for any $\mu\in \h^\ast$ such that $L_{\mu}\in \mc O_{\la}$. This implies that   $\Tw_{s_\alpha} :\mc O[{\chi_\la}] \rightarrow \mc O[{\chi_{\la}}]$  restricts to  an equivalence from $\mc O_\la$ to $\mc O_{s_\alpha\la}$. This completes the proof.
\end{proof}


\subsection{Equivalence of categories: the case $\la$ generic}  \label{Sect::Eqivxyz} Let $\la=(x,y,z)$ be a generic atypical weight so that $(\la,(1,\sigma,\tau))=0$, for some $\sigma,\tau\in\{\pm1\}$. In this section, we shall show that $\mc O_\la$ is equivalent to the principal block of $\gl(1|1)$.

By applying the twisting functor $\Tw_{s_1}$ and/or $\Tw_{s_2}$ in Lemma \ref{lem::equiTs} to $\OO_\la$, if necessary, we can assume that $(\la,\alpha_0)=0$. It follows that if $L_\mu$ lies in $\OO_\la$, then $\mu\in\la+\Z\alpha_0$ by Theorem \ref{thm::genericthm}.

We embed $\gl(1|1)$ into $\g$ as follows:
\begin{align*}
&E_{-1,1}\rightarrow e_{\delta-\ep_1-\ep_2},\quad E_{1,-1}\rightarrow f_{\delta-\ep_1-\ep_2},\\
&E_{-1,-1}+E_{11}\rightarrow \frac{1+\zeta}{2}h_{2\delta}+\frac{1}{2}h_{2\ep_1}+\frac{\zeta}{2}h_{2\ep_2},\\
&E_{-1,-1}-E_{11}\rightarrow 2h_{2\delta}.
\end{align*}

We identify $\gl(1|1)$ inside $D(2|1;\zeta)$ via this embedding. Let $\mf l$ be the Levi subalgebra corresponding to $\{\alpha_0\}\subseteq\Pi$ and $\mf p=\mf l+\mf u$ be the corresponding parabolic subalgebra with radical $\mf u$. Then $\mf l$ can be identified with $\mf l=\gl(1|1)\oplus\C \left(h_{2\delta}+h_{2\ep_2}\right)$. The weight $\la=(x,y,z)$ translates to the weight $ x\delta^{\mf a}_{-1}- x\delta^{\mf a}_1$ on $\h_{1|1}\subseteq\gl(1|1)\subseteq\mf l$. Furthermore, if $\mu=\la+k\alpha_0$, $k\in\Z$, then $\mu$ translates to the weight $(x+k)\delta_{-1}^{\mf a}-(x+k)\delta^{\mf a}_1$ on $\h_{1|1}$.

\begin{thm}\label{thm:eqv:generic}
Let $\la=(x,y,z)$ be generic and suppose that $(\la,\alpha_0)=0$. The parabolic induction functor $\rm{Ind}_{\mf p}^\g:\OO(\gl(1|1),\mf b^s)\rightarrow\OO$ restricts to an equivalence of highest weight categories from $\OO(\gl(1|1),\mf b^s)_0$ to $\OO_\la$.
\end{thm}

\begin{proof}
The induction functor clearly sends Verma modules to Verma modules. Let $\mu\in\la+\Z\alpha_0$. Interpreting $\mu$ as a weight of $\gl(1|1)$ as above, we let $L^{\mf a}(\mu)$ be the irreducible $\gl(1|1)$-module of highest weight $\mu$. We shall now show that it sends irreducibles to irreducibles, i.e., ${\rm Ind}_{\mf l}^{\mf p}L^{\mf a}(\mu)$ is irreducible. To see this, observe that if $\nu$ is the weight of a non-zero singular vector in ${\rm Ind}_{\mf l}^{\mf p}L^{\mf a}(\mu)$, then $\nu\in\la+\Z\alpha_0$, and hence $\nu\in\mu+\Z\alpha_0$. But in this case $\nu=\mu-k\alpha_0$, and hence must be a weight in $L^{\mf a}(\mu)$. As $L^{\mf a}(\mu)$ is irreducible, $\nu=\mu$ and hence ${\rm Ind}_{\mf l}^{\mf p}L^{\mf a}(\mu)$ is irreducible.

Now, for $M\in\OO$ consider the space of $\mf u$-invariants $M^{\mf u}$. Then $M^{\mf u}$ is an $\mf l$-module, and thus taking the $\mf u$-invariants defines a functor from $\OO$ to $\OO(\gl(1|1))$. The exact same argument as in the proof of \cite[Proposition 3.6]{CMW13} shows that this invariant functor restricted to $\OO_\la$ gives the inverse equivalence of the induction functor restricted on the principal block of $\gl(1|1)$.
\end{proof}

In particular, Theorem \ref{thm:eqv:generic} and \eqref{gl11:char} imply that the following character formula for generic atypical $\la$ with $(\la,\alpha)=0$ and $\alpha>0$:
\begin{align*}
{\rm ch}L_\la= \sum_{i=0}^\infty (-1)^i{\rm ch}M_{\la-i\alpha}.
\end{align*}

\subsection{Equivalence of categories: the case $\la=(0,\pm \zeta z,z)$}  \label{Sect::Eqiv0yz}

In this subsection, we consider  atypical weights of the form $\la=(0,\pm \zeta z,z)$ with $\zeta z,z\not\in\Z$. Again, thanks to Lemma \ref{lem::equiTs} we have equivalence $\mc O_\la \cong \mc O_{s_{1}\la}$. Therefore we assume  that $\la=(0,  \zeta z,z)$.
Let $\beta=\delta+\ep_1-\ep_2$ and $\gamma=\delta-\ep_1+\ep_2$.

Applying odd reflection with respect to the root $\alpha_0=\delta-\ep_1-\ep_2$ to the standard Dynkin diagram of $D(2|1;\zeta)$ we obtain the following Dynkin diagram $\Pi'$:
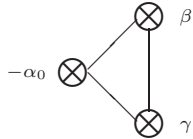
\begin{center}
	\setlength{\unitlength}{0.16in}
	\begin{picture}(4,6)
	\put(4,1.3){\makebox(0,0)[c]{$\bigotimes$}}
	\put(4,4.8){\makebox(0,0)[c]{$\bigotimes$}}
	\put(1.5,3){\makebox(0,0)[c]{$\bigotimes$}}
	\put(3.6,1.4){\line(-1,1){1.6}}
	\put(3.6,4.7){\line(-1,-1){1.6}}
    \put(4,1.7){\line(0,1){2.7}}
	\put(5.2,4.8){\makebox(0,0)[c]{\tiny $\beta$}}
	\put(5.2,1.2){\makebox(0,0)[c]{\tiny $\gamma$}}
	\put(0,3){\makebox(0,0)[c]{\tiny $-\alpha_0$}}
	\end{picture}
\end{center}

Denote the BGG category of $D(2|1;\zeta)$ with respect to the simple system $\Pi'$ by $\OO'$. For $\la\in\h^*$ denote by $M'_\la$, $L'_\la$, and $P'_\la$, the corresponding Verma, irreducible and indecomposable projective modules in $\OO'$. Recall that we are using the $\rho$-shifted notation.

The categories $\OO$ and $\OO'$ are equivalent as abelian categories, but not as highest weight categories. The identity functor indeed gives such an equivalence from $\OO$ to $\OO'$, which we shall refer to as the odd reflection functor associated with the odd root $\alpha_0$.

By Theorem \ref{Prop::IrrBlocks} we have that if $L_\mu\in\OO_\la$, then $(\mu,\alpha_0)\not=0$. Now the argument in the proof of \cite[Proposition 3.8]{CMW13} shows that the odd reflection functor corresponding to $\alpha_0$ restricted to the block $\OO_\la$ sends $M_\la$ to $M'_\la$. Since it sends $L_\la$ to $L'_\la$, and hence $P_\la$ to $P'_\la$ as well, we see that the odd reflection functor gives an equivalence of highest weight categories between the blocks $\OO_\la$ and $\OO'_\la$.

The simple roots $\{\beta,\gamma\}$ of the system $\Pi'$ gives rise a Levi subalgebra $\mf l$ with parabolic subalgebra $\mf p$. Note that $\mf l$ is isomorphic to $\gl(2|1)$ in its non-distinguished form. The parabolic induction functor defines an exact functor $\text{Ind}_{\mf p}^\g:\OO(\gl(2|1),\mf b)\rightarrow\OO'$. Recall that $\OO(\gl(2|1),\mf b)_0$ denotes the principal block of $\OO(\gl(2|1),\mf b)$. We have the following.

\begin{thm}\label{thm:equiv:1}
Let $\la=(0,\zeta z,z)$. The functor $\rm{Ind}_{\mf p}^\g$ restricts to an equivalence of highest weight categories from $\OO(\gl(2|1),\mf b)_0$ to $\OO'_\la$. Consequently, we have $\OO(\gl(2|1),\mf b)_0\cong\OO_\la$ as highest weight categories.
\end{thm}

\begin{proof}
The functor $\rm{Ind}_{\mf p}^\g$ sends $\mf l$-Verma modules to $\g$-Verma modules. Thus, as in the proof of Theorem \ref{thm:eqv:generic}, in order to complete the proof it suffices to show that it sends irreducibles to irreducibles.

To do this we realize an explicit embedding of $\gl(2|1)$ into $D(2|1;\zeta)$ as follows. Recalling the notations in \eqref{D21basis} and \eqref{gl21coroots} we define the following linear map from $\mf{sl}(2|1)$ to $D(2|1;\zeta)$:
\begin{align*}
&e_{\delta^{\mf a}_{-2}-\delta^{\mf a}_1}\rightarrow \frac{1}{1+\zeta} e_{\delta+\ep_1-\ep_2}, \quad e_{\delta^{\mf a}_1-\delta^{\mf a}_{-1}}\rightarrow -\frac{1}{1+\zeta} e_{\delta-\ep_1+\ep_2},\\
&f_{\delta^{\mf a}_{-2}-\delta^{\mf a}_1}\rightarrow f_{\delta+\ep_1-\ep_2}, \quad f_{\delta^{\mf a}_1-\delta^{\mf a}_{-1}}\rightarrow -f_{\delta-\ep_1+\ep_2},\\
&h_{\delta^{\mf a}_{-2}-\delta^{\mf a}_1}\rightarrow \frac{1}{2} h_{2\delta}-\frac{1}{2(1+\zeta)}h_{2\ep_1}+\frac{\zeta}{2(1+\zeta)}h_{2\ep_2},\\
&h_{\delta^{\mf a}_1-\delta^{\mf a}_{-1}}\rightarrow \frac{1}{2} h_{2\delta}+\frac{1}{2(1+\zeta)}h_{2\ep_1}-\frac{\zeta}{2(1+\zeta)}h_{2\ep_2}.
\end{align*}

Using the formulas in \eqref{D21basis} and \eqref{gl21coroots} we can verify that
\begin{align*}
&e_{\delta^{\mf a}_{-2}-\delta^{\mf a}_{-1}}\rightarrow e_{2\delta},\quad f_{\delta^{\mf a}_{-2}-\delta^{\mf a}_{-1}}\rightarrow f_{2\delta},\quad h_{\delta^{\mf a}_{-2}-\delta^{\mf a}_{-1}}\rightarrow h_{2\delta}.
\end{align*}
Thus, this is indeed an embedding of $\mf{sl}(2|1)$ into $D(2|1;\zeta)$. To get an embedding from $\gl(2|1)$ in $D(2|1;\zeta)$ we send the identity matrix in $\gl(2|1)$ to $h_{2\ep_1}+h_{2\ep_2}$. Identifying $\gl(2|1)$ inside $D(2|1;\zeta)$ via this embedding the weight $\la=(0,\zeta z,z)$ translates to the weight $(1+\zeta)z\delta^{\mf a}_{-2}+(1+\zeta)z\delta^{\mf a}_{-1}-(1+\zeta)z\delta^{\mf a}_1$ on $\h_{2|1}$.

Let $\mu=w(\la+k\beta)\in W_\la(\la+\Z\beta)$. We compute that $\mu$ transforms as $w(k\delta^{\mf a}_{-2}-k\delta^{\mf a}_1)+(1+\zeta)z{\rm Str}$ on $E_{ii}$, for $i=1,2,3$, where $w$ in the latter expression is regarded as an element in the Weyl group $W^{\mf a}$ of $\gl(2|1)$ in a natural way, i.e., $w$ is interpreted as the transposition permuting $\delta^{\mf a}_{-2}$ and $\delta^{\mf a}_{-1}$, if $w$ is non-trivial, and $w$ is interpreted as the identity otherwise.

Note that the set $\{w(k\delta^{\mf a}_{-2}-k\delta^{\mf a}_1)|w\in W^{\mf a}\}$ exhausts the complete set of irreducible highest weights in the principal block of $\gl(2|1)$. Hence, the set $W_\la(\la+\Z\beta)$, when restricted to $E_{ii}$, $i=-2,-1,1$, gives a complete set of irreducible highest weights in the principal block of $\gl(2|1)$ (tensored with the one-dimensional representation $(1+\zeta)z{\rm Str}$).

For $\mu\in W_\la(\la+\Z\beta)$, denote by $L^{\mf a}(\mu)$ the irreducible $\gl(2|1)$-module of highest weight $\mu$ with the respect to the non-standard Borel subalgebra. Here we regard $\mu$ as a $\gl(2|1)$-highest weight as explained above. We consider the $\g$-module ${\rm Ind}_{\mf p}^{\g}L^{\mf a}({\mu})$. We claim that
\begin{align*}
L_\mu\cong {\rm Ind}_{\mf p}^{\g}L^{\mf a}({\mu}).
\end{align*}
To see this, assume $\nu$ is a weight for a singular vector in ${\rm Ind}_{\mf p}^{\g}L^{\mf a}({\mu})$. Then we must have $\nu\in W_\la(\la+\Z\beta)$. However, any such weight in the induced module is of the form $\mu-m\beta-n\gamma$, for some $m,n\in\Z_{\ge 0}$. Now, such a weight must be a weight in $L^{\mf a}(\mu)$, which is irreducible, and hence the weight can only be $\mu$ itself. This proves that ${\rm Ind}_{\mf p}^{\g}L^{\mf a}({\mu})$ is irreducible and hence isomorphic to $L_\mu$.
\end{proof}

By Theorems \ref{thm:equiv:1} and \ref{char:gl21} we obtain a closed character formula for the irreducible modules in the block $\OO_\la$ by replacing $(0,0|0)$ by $\la$, $(1,0|-1)$ by $\beta$, and $(0,-1|1)$ by $\gamma$. Explicitly, we have the following.

\begin{thm}\label{thm:char:xint} Let $\la=(0,\zeta z,z)$ with $z,\zeta z\not\in\Z$ and let $\beta=\delta+\ep_1-\ep_2$ and $\gamma=\delta-\ep_1+\ep_2$.
The irreducible modules in the block $\OO_\la$ in terms of Verma modules with respect to the standard Borel subalgebra are given follows:
\begin{align*}
&L_{\la+n\beta}=\sum_{i=-\infty}^n(-1)^{n+i} M_{\la+i\beta}+\sum_{j=-\infty}^{-n-1}(-1)^{n+j}M_{\la+j\gamma}, \quad n\ge 0,\\
&L_{\la+n\gamma}=\sum_{i=-\infty}^n(-1)^{n+i} M_{\la+i\gamma}+\sum_{j=-\infty}^{-n-1}(-1)^{n+j}M_{\la+j\beta}, \quad n\ge 0,\\
&L_{\la-n\beta}=\sum_{i=-\infty}^{-n}(-1)^{n+i} M_{\la+i\beta}, \quad n> 0,\\
&L_{\la-n\gamma}=\sum_{i=-\infty}^{-n}(-1)^{n+i} M_{\la+i\gamma}, \quad n> 0.
\end{align*}
\end{thm}

\subsection{Equivalence of categories: the case $\la=(x,-(1+\zeta)x,0)$}  \label{Sect::Eqivxy0}

In this subsection, we consider  atypical weights of the form $\la=(x,\pm(1+\zeta)x,0)$ with $x,(1+\zeta)x\not\in\Z$. Again, thanks to Lemma \ref{lem::equiTs} we have equivalence $\mc O_\la \cong \mc O_{s_{1}\la}$. Therefore we assume  that   $\la=(x,-(1+\zeta)x ,0)$.

Let $\phi:D(2|1;-1-\zeta)\rightarrow D(2|1;\zeta)$ be the isomorphism in Section \ref{sec:D:iso2}. If $\mu=(x,y,z)$ and $L_\mu$ is the highest weight irreducible module of $D(2|1;-1-\zeta)$ of highest weight $\mu$ in the block $\OO_{(0,-(1+\zeta)x,x)}$ with respect to the standard Borel subalgebra, then the pullback under $\phi^{-1}$ is an irreducible $D(2|1;\zeta)$-module of highest weight $(z,y,x)$ in the block $\OO'_{(x,-(1+\zeta)x,0)}$, the highest weight category with respect to the Borel subalgebra $\mf b'$ corresponding to the simple system $\Pi'$ of Section \ref{sec:oo'}. Thus, we have an equivalence of highest categories between $\OO_{(x,-(1+\zeta)x,0)}$ for $D(2|1;\zeta)$ with $\OO'_{(0,-(1+\zeta)x,x)}$ for $D(2|1;-1-\zeta)$. Clearly, there is an equivalence of categories between $\OO'_{(0,-(1+\zeta)x,x)}$ for $D(2|1;-1-\zeta)$ with $\OO_{(0,-(1+\zeta)x,x)}$ for $D(2|1;-1-\zeta)$. Thus, the computation of irreducible characters in Theorem \ref{thm:char:xint} also implies the irreducible characters in the block $\OO_\la$. We shall illustrate this with an example below.

\begin{example}\label{example1}
Let $\la=(x,-(1+\zeta)x,0)$ and suppose we want to compute the character of the irreducible $D(2|1;\zeta)$-module $L_{\la}$. We shall need to know the highest weight of $L_\la$ with respect to the Borel $\mf b'$ associated with the simple system $\Pi'$. This is computed by the formulas in \cite[Lemma 1]{PS89} (see also \cite[Lemma 1.40]{CW12}).
To be precise, $\Pi'$ is obtained from the standard system $\Pi$ by applying first the odd reflection corresponding to $\alpha_0=\delta-\ep_1-\ep_2$, and then the odd reflection corresponding to $\beta=\delta+\ep_1-\ep_2$. We compute $(\la,\alpha_0)=0$ and thus, with respect to the new Borel the shifted highest weight becomes $\la+\alpha_0$. Now, we compute $(\la+\alpha_0,\delta+\ep_1-\ep_2)\not=0$, and hence the highest weight with respect to $\mf b'$ is $\la+\alpha_0=(x+1,-(1+\zeta)x-1,-1)$. To compute this character is equivalent to compute the character of $L_{(-1,-(1+\zeta)x-1,x+1)}=L_{(0,-(1+\zeta)x,x)-\beta}$, which is computed in Theorem \ref{thm:char:xint} with $\zeta$ replaced by $-1-\zeta$.
\end{example}

\subsection{Equivalence of categories: the case $\la=(x,0,-\frac{1+\zeta}{\zeta}x)$} \label{Sect::Eqivx0z}

In this subsection, we consider  atypical weights of the form $\la=(x,0,\pm\frac{1+\zeta}{\zeta}x)$ with $x,-\frac{1+\zeta}{\zeta}x\not\in\Z$. Again, thanks to Lemma \ref{lem::equiTs} we have equivalence $\mc O_\la \cong \mc O_{s_{2}\la}$. Therefore we assume  that $\la=(x,0,-\frac{1+\zeta}{\zeta}x)$.

Let $\phi:D(2|1;\frac{1}{\zeta})\rightarrow D(2|1;\zeta)$ be the isomorphism in Section \ref{sec:D:iso1}. If $\mu=(x,y,z)$ and $L_\mu$ is the highest weight irreducible module of $D(2|1;\frac{1}{\zeta})$ of highest weight $\mu$ in the block $\OO_{(x,-(1+\frac{1}{\zeta})x,0)}$ with respect to the standard Borel subalgebra, then the pullback under $\phi^{-1}$ is an irreducible $D(2|1;\zeta)$-module of highest weight $(x,z,y)$ in the block $\OO_{(x,0,-\frac{1+\zeta}{\zeta}x)}$. The isomorphism $\phi$ induces an equivalence of highest categories between $\OO_{(x,0,-\frac{1+\zeta}{\zeta}x)}$ for $D(2|1;\zeta)$ and $\OO_{(x,-(1+\frac{1}{\zeta})x,0)}$ for $D(2|1;\frac{1}{\zeta})$. Thus, the computation of character for the irreducible $D(2|1;\frac{1}{\zeta})$-modules of $\OO_{(x,-(1+\frac{1}{\zeta})x,0)}$ in Section \ref{Sect::Eqivxy0} also computes the characters of the irreducible $D(2|1;\zeta)$-modules in the block $\OO_\la$, for $\la=(x,0,-\frac{1+\zeta}{\zeta}x)$.

\vskip 1cm

\section{Character formulas in $2$-integer case} \label{Sect::ChFormulae}


Section \ref{sec:reduction} gives the character of $L_\la$ in $\OO$ in the case when $\la=(x,y,z)$ is an atypical weight with at most one among the $\{x,y,z\}$ being an integer. In this section we shall give the character of $L_\la$ in the case when two among the $\{x,y,z\}$ are integers. To do that we compute, equivalently, the character of the tilting module $T_\la$, for such $\la$. The method of computation for these tilting characters follows closely the one for those of integral highest weights in \cite{CW19}. That is, we apply suitable translation functors to tilting module with known character formulas, e.g., tilting modules with typical highest weights. The resulting module is easily seen to be a direct sum of tilting modules. Using Proposition \ref{prop:flags} and with a careful choice of translation functor, we indeed show that the resulting module is indeed indecomposable and hence a tilting module. Our calculations in fact show that Proposition \ref{prop:flags} is not just a necessary condition, indeed in the setting of this present paper it is sufficient as well! This enables us to write closed formulas for the tilting modules in almost all cases. Below we shall write down these formulas explicitly. 

In this  section we assume that $\la=(x,y,z)$ is an atypical non-integral weight.

\subsection{The case of $\la=(x,y,z)$ with $x\not\in\Z$ and $y,z\in\Z$}\label{xnotinz}

 Thanks to Theorem \ref{Prop::IrrBlocks} and Lemma \ref{lem::equiTs}, we   assume in this section that $x$ is positively proportional to $\frac{\ka}{1+\ka}$ modulo an integer.

Let $\ell\ge 1$. We define the following weight for each $k\in\Z$:
\begin{align*}
{}_\ell\la_k=(\frac{\ell\zeta}{\zeta+1}+k,-|k|,-|k+\ell|).
\end{align*}
Then ${}_\ell\la_k$ is anti-dominant. Let $W'\cong \Z_2\times\Z_2$ be subgroup of the Weyl group changing the signs of the last two coordinates. That is, $W'$ is the integral Weyl group of ${}_\ell\la_k$, for any $k$. We shall use the following simplified notations for the Weyl group conjugates of ${}_\ell\la_k$:
\begin{align*}
&{}_\ell\la_k^1=(\frac{\ell\zeta}{\zeta+1}+k,|k|,-|k+\ell|),\\
&{}_\ell\la_k^2=(\frac{\ell\zeta}{\zeta+1}+k,-|k|,|k+\ell|),\\
&{}_\ell\la_k^{12}=(\frac{\ell\zeta}{\zeta+1}+k,|k|,|k+\ell|).
\end{align*}

The following is straightforward to verify.

\begin{lem}
If   $x\not\in\Z$ and $y,z\in\Z$, then $\la\in\{w'{}_\ell\la_k|k\in\Z,w\in W'\}$.
\end{lem}

In the sequel, when it is clear from the context we shall freely drop $\ell$ from the formula in order to simplify notation. For example, we shall write $\la_k$ for ${}_\ell\la_k$ etc. This convention applies to later sections as well.

\begin{thm}\label{thm:aux1} Suppose $k\not=1, -\ell+1$. Suppose $\la\in\{w'{}_\ell\la_k|k\in\Z,w\in W'\}$ is  atypical. Then we have the following character formula for the tilting module of highest weight $\la-\rho$:
\begin{align}\label{char:std}
T_\la=\sum_{\mu\le\la}M_\mu,
\end{align}
where the summation above is over those weights $\mu$ lying in the set $\{w'\la,w'(\la-\alpha)|w'\in W', \alpha\in\Phi^+_{\bar 1}, (\la,\alpha)=0\}$ such that $\mu\le\la$. In particular, $T_\la$ is multiplicity-free.
\end{thm}

\begin{thm}\label{thm:aux2} Suppose that $\ell>1$ and $k=1,-\ell+1$. Then we have the following:
\begin{align*}
&T_{\la_{1}}=M_{\la_{1}}+M_{\la_0}+M_{\la_{-1}},\\
&T_{\la_{1}^2}=M_{\la_{1}^2}+M_{\la_0^2}+M_{\la_{-1}^2}+M_{\la_{1}}+M_{\la_0}+M_{\la_{-1}},\\
&T_{\la_{-\ell+1}}= M_{\la_{-\ell+1}}+M_{\la_{-\ell}}+M_{\la_{-\ell-1}},\\
&T_{\la_{-\ell+1}^1}= M_{\la_{-\ell+1}^1}+M_{\la_{-\ell}^1}+M_{\la_{-\ell-1}^1} + M_{\la_{-\ell+1}}+M_{\la_{-\ell}}+M_{\la_{-\ell-1}}.
\end{align*}
Furthermore, for other $\la$s, i.e., $\la=\la_{1}^1, \la_{1}^{12},\la_{-\ell+1}^2,\la_{-\ell+1}^{12}$, the standard formula \eqref{char:std} holds.
\end{thm}

\begin{thm}\label{thm:aux3} For $\ell=1$ and $k=0,1$ we have the following formulas for the tilting modules:
\begin{align*}
&T_{\la_{0}}= M_{\la_{0}}+M_{\la_{-1}^1}+M_{\la_{-1}}+M_{\la_{-2}^1}+M_{\la_{-2}},\\
&T_{\la_{1}}= M_{\la_{1}}+M_{\la_{0}}+M_{\la_{-1}},\\
&T_{\la_{1}^2}= M_{\la_{1}^2}+M_{\la_{1}}+M_{\la_{0}^2}+M_{\la_{0}}+M_{\la_{-1}}.
\end{align*}
For other $\la$s, i.e., $\la=\la_{0}^2, \la_{1}^1,\la_{1}^{12}$, the standard formula \eqref{char:std} holds.
\end{thm}

For $\la\in\mathfrak{h}^*$, we let $\mathcal{E}_\la$ denote the translation functor given by tensoring with the adjoint $\mathfrak{g}$-module and then projecting to $\mathcal{O}_\la$. Below, we shall sketch a proof of Theorems \ref{thm:aux1}-- \ref{thm:aux3}.

\begin{proof}[Proof of Theorems \ref{thm:aux1}-- \ref{thm:aux3}]

Take $\la=(x,y,z)$ with $x\not\in \mathbb{Z}$ and $y,z\in\mathbb{Z}$. Suppose $t$ is the minimal positive integer such that $\la-2t\delta$ is typical. We will prove the formulas by induction on $t$.

If $t=1$, then $\la-2\delta$ is typical and hence
$$T_{\la-2\delta}=
\left\{
\begin{array}{ll}
M_{\la-2\delta}, & \mbox{if $\la={}_\ell\la_k$},\\
M_{\la-2\delta}+M_{s_1(\la-2\delta)}, & \mbox{if $\la={}_\ell\la_k^1$ and $k\neq0$},\\
M_{\la-2\delta}+M_{s_2(\la-2\delta)}, & \mbox{if $\la={}_\ell\la_k^2$ and $k+\ell\neq0$},\\
M_{\la-2\delta}+M_{s_1(\la-2\delta)}+M_{s_2(\la-2\delta)}+M_{s_1s_2(\la-2\delta)}, & \mbox{if $\la={}_\ell\la_k^{12}$ and $k,k+\ell\neq0$},
\end{array}
\right.$$
where $s_1$ and $s_2$ are defined in \eqref{def:s}. We compute $\mathcal{E}_\la(T_{\la-2\delta})$ and see that
it equals the right hand side of the formula for $T_\la$ in Theorems \ref{thm:aux1}-- \ref{thm:aux3}. Thus, we must have $\mathcal{E}_\la(T_{\la-2\delta})=T_\la$ thanks to Corollary \ref{coro:23}. Therefore the formulas in Theorems \ref{thm:aux1}--\ref{thm:aux3} hold in this case.

If $t>1$, we consider the atypical weight $\la-2\delta$, which is also $2$-integer with the first coordinate non-integral. Let $\la$ and $\la-2\delta$ have atypical roots $(1,\sigma,\tau)$ and $(1,\sigma',\tau')\in\Phi_{\overline{1}}$, respectively. It is clear that $(1,\sigma,\tau)\neq(1,\sigma',\tau')$.
By inductive assumption, we know
$$T_{\la-2\delta}=
\left\{
\begin{array}{ll}
M_{\la-2\delta}+\mbox{lower terms}, \hskip5.1cm \mbox{if $\la={}_\ell\la_k$},\\
M_{\la-2\delta}+M_{s_1(\la-2\delta)}+\mbox{lower terms}, \hskip3cm \mbox{if $\la={}_\ell\la_k^1$ and $k\neq0$},\\
M_{\la-2\delta}+M_{s_2(\la-2\delta)}+\mbox{lower terms}, \hskip3cm \mbox{if $\la={}_\ell\la_k^2$ and $k+\ell\neq0$},\\
M_{\la-2\delta}+M_{s_1(\la-2\delta)}+M_{s_2(\la-2\delta)}+M_{s_1s_2(\la-2\delta)}+\mbox{lower terms},\\
\hskip8.7cm \mbox{if $\la={}_\ell\la_k^{12}$ and $k,k+\ell\neq0$.}
\end{array}
\right.$$
where the weights appearing in lower terms are always in the form of $\mu=(x-3,\pm (y-\sigma'),\pm (z-\tau'))$, or $(x-4,y,\pm (z-2\tau'))$ in case of $y=\sigma'=-1$, or $(x-4,\pm (y-2\sigma'), z)$ in case of $z=\tau'=-1$. For any such $\mu$ and any $\alpha\in\Phi$, we have $\mu+\alpha\not\in\mathrm{Irr}\OO_{\la}$ by Theorem \ref{Prop::IrrBlocks}. Thus, applying $\mathcal{E}_{\la}$ to $T_{\la-2\delta}$, we still have that $\mathcal{E}_\la(T_{\la-2\delta})$ equals the right hand side of the formula for $T_\la$ in Theorems \ref{thm:aux1}-- \ref{thm:aux3}. Hence we must have $\mathcal{E}_\la(T_{\la-2\delta})=T_\la$ by Corollary \ref{coro:23} again. Thus, the formulas in Theorem \ref{thm:aux1} -- \ref{thm:aux3} hold.
\end{proof}

\subsection{The case of $\la=(x,y,z)$ with $z\not\in\Z$ and $x,y\in\Z$}\label{znotinZ}
  Thanks to Theorem \ref{Prop::IrrBlocks} and Lemma \ref{lem::equiTs}, we   assume in this section  that $z$ is positively proportional to $\frac{1}{\ka}$  modulo an integer.

Fix $\ell\ge 1$. For $k\in\Z$ we define the following anti-dominant weight:
\begin{align*}
{}_\ell\la_k=(-|k|,-|k-\ell|,k+\frac{\ell}{\zeta}).
\end{align*}

Let $W'\cong \Z_2\times\Z_2$ be subgroup of the Weyl group changing the signs of the first two coordinates. That is, $W'$ is the integral Weyl group of ${}_\ell\la_k$, for any $k$. We shall use the following simplified notations for the Weyl group conjugates of ${}_\ell\la_k$:
\begin{align*}
&{}_\ell\la_k^0=(|k|,-|k-\ell|,k+\frac{\ell}{\zeta}),\\
&{}_\ell\la_k^1=(-|k|,|k-\ell|,k+\frac{\ell}{\zeta}),\\
&{}_\ell\la_k^{01}=(|k|,|k-\ell|,k+\frac{\ell}{\zeta}).
\end{align*}

\begin{lem}
If   $z\not\in\Z$ and $x,y\in\Z$, then $\la\in\{w'{}_\ell\la_k|k\in\Z,w\in W'\}$.
\end{lem}

\begin{rem} The character formulas in this section can be obtained by two methods: We can compute directly using suitable translation functors applied to tilting modules with known character formulas as in Section \ref{xnotinz}. Alternatively, we can first compute the characters of the irreducible $D(2|1;-1-\zeta)$-modules in the block $\OO_{(\frac{1+\zeta}{\zeta}\ell,0,-\ell)}$ using Theorems \ref{thm:aux1}--\ref{thm:aux3} and Theorem \ref{BGGS}. As in Example \ref{example1}, one then computes the characters of the irreducible $D(2|1;-1-\zeta)$-modules in $\OO'_{(\frac{1+\zeta}{\zeta}\ell,0,-\ell)}$. Now, the isomorphism of Section \ref{sec:D:iso2} relates the $D(2|1;\zeta)$-irreducible modules in the block $\OO_{{}_\ell\la_0}$ to $D(2|1;-1-\zeta)$-irreducible modules in $\OO'_{(\frac{1+\zeta}{\zeta}\ell,0,-\ell)}$ in the exact same way as in Section \ref{Sect::Eqivxy0}.
\end{rem}

For completeness we write down the character of the tilting modules.

\begin{thm}\label{thm:aux4} Suppose that $k\not=\ell-1,\ell+1$. Suppose $\la\in\{w'{}_\ell\la_k|k\in\Z,w\in W'\}$ is atypical. Then the formula \eqref{char:std} holds.
\end{thm}

\begin{thm}\label{thm:aux5} Suppose that $\ell>1$ and $k=\ell-1,\ell+1$. Then we have the following:
\begin{align*}
&T_{\la_{\ell-1}}= M_{\la_{\ell-1}}+M_{\la_{\ell}}+M_{\la_{\ell+1}},\\
&T_{\la_{\ell+1}^0}=M_{\la_{\ell+1}^0}+M_{\la_{\ell+1}}+M_{\la_{\ell}^0}+M_{\la_{\ell}}+M_{\la_{\ell-1}^0}+M_{\la_{\ell-1}}.
\end{align*}
For other $\la$s, the formula \eqref{char:std} holds.
\end{thm}

\begin{thm}\label{thm:aux6} For $\ell=1$ and $k=0,2$, we have the following formulas for the tilting modules:
\begin{align*}
&T_{\la_{0}}= M_{\la_{0}}+M_{\la_{-1}}+M_{\la_{1}}+M_{\la_{2}},\\
&T_{\la_{2}^0}= M_{\la_{2}^0}+M_{\la_{2}}+M_{\la^{0}_{1}} + M_{\la_{1}}+M_{\la_{0}}.
\end{align*}
Furthermore, for other $\la$s, the standard formula \eqref{char:std} holds.
\end{thm}

\subsection{The case of $\la=(x,y,z)$ with $y\not\in\Z$ and $x,z\in\Z$} \label{ynotinZ}

 Thanks to Theorem \ref{Prop::IrrBlocks} and Lemma \ref{lem::equiTs}, we   assume in this section  that $y$ is positively proportional to $\ell \ka$ modulo an integer.

Fix $\ell\ge 1$. For $k\in\Z$ we define the following anti-dominant weight:
\begin{align*}
{}_\ell\la_k=(-|k|,k+\ell\zeta,-|k-\ell|).
\end{align*}

Let $W'\cong \Z_2\times\Z_2$ be subgroup of the Weyl group changing the signs of the first and the last coordinates. That is, $W'$ is the integral Weyl group of ${}_\ell\la_k$, for any $k$. We shall use the following simplified notations for the Weyl group conjugates of ${}_\ell\la_k$:
\begin{align*}
&{}_\ell\la_k^0=(|k|,k+\ell\zeta,-|k-\ell|),\\
&{}_\ell\la_k^2=(-|k|,k+\ell\zeta,|k-\ell|),\\
&{}_\ell\la_k^{02}=(|k|,k+\ell\zeta,|k-\ell|).
\end{align*}

The following is straightforward.

\begin{lem}
If   $y\not\in\Z$ and $x,z\in\Z$, then $\la\in\{w'{}_\ell\la_k|k\in\Z,w\in W'\}$.
\end{lem}

\begin{rem} The characters of the tilting modules in this section are obtained as follows. Recall that the isomorphism in Section \ref{sec:D:iso1} induces an equivalence of highest weight categories between the blocks $\OO_{{}_\ell\la_0}$ of $D(2|1;\zeta)$-modules and the block $\OO_{(0,-\ell,\ell\zeta)}$ of $D(2|1;\frac{1}{\zeta})$-modules. Thus, tilting character of the latter, computed in Theorems \ref{thm:aux4}--\ref{thm:aux6}, also computes the tilting character of the former.
\end{rem}

For completeness we write down the character of the tilting modules.

\begin{thm} Suppose that $k\not=\ell-1,\ell+1$. Suppose $\la\in\{w'{}_\ell\la_k|k\in\Z,w\in W'\}$ is atypical. Then the formula \eqref{char:std} holds.
\end{thm}

\begin{thm} Suppose that $\ell>1$ and $k=\ell-1, \ell+1$. Then we have the following:
\begin{align*}
&T_{\la_{\ell-1}}= M_{\la_{\ell-1}}+M_{\la_{\ell}}+M_{\la_{\ell+1}},\\
&T_{\la_{\ell+1}^0}=M_{\la_{\ell+1}^0}+M_{\la_{\ell+1}}+M_{\la_{\ell}^0}+M_{\la_{\ell}}+M_{\la_{\ell-1}^0}+M_{\la_{\ell-1}}.
\end{align*}
For other $\la$s, the formula \eqref{char:std} holds.
\end{thm}

\begin{thm} For $\ell=1$ and $k=0,2$, we have the following formulas for the tilting modules:
\begin{align*}
&T_{\la_{0}}= M_{\la_{0}}+M_{\la_{-1}}+M_{\la_{1}}+M_{\la_{2}},\\
&T_{\la_{2}^0}= M_{\la_{2}^0}+M_{\la_{2}}+M_{\la^{0}_{1}} + M_{\la_{1}}+M_{\la_{0}}.
\end{align*}
Furthermore, for other $\la$s, the standard formula \eqref{char:std} holds.
\end{thm}

\begin{rem}
For simple Lie algebras it is known that the irreducible characters of non-integral highest weight modules can be computed by the irreducible characters of integral highest weight modules of smaller rank simple Lie algebras, see, e.g., \cite{Lu84, Soe90}.  For Lie superalgebras, such an {\em integral reduction} to smaller rank is only known for the general linear Lie superalgebras \cite{CMW13}, the queer Lie superalgebras \cite{Ch16}, and some other very special cases \cite{GK15}. Comparing the characters in Sections \ref{xnotinz}--\ref{ynotinZ} with the integral blocks of classical simple Lie superalgebras with degree atypicality one, one can show that the $2$-integer blocks of $D(2|1;\zeta)$ are not equivalent to any of such integral blocks. Thus, these non-integral blocks of $D(2|1;\zeta)$ provide examples where the integral reductions to smaller rank simple basic classical Lie superalgebras fail. On the other hand, in the case when at most one among the $x,y,z$ is an integer we have integral reduction to smaller rank basic Lie superalgebras as illustrated in Section \ref{sec:reduction} and suggested by \cite{GK15}.
\end{rem}

\section{Primitive spectrum of $\D$} \label{Section::Prim}
As an application of the results in Sections  \ref{sec:reduction} and \ref{Sect::ChFormulae}, we   prove that the inclusions of primitive ideals in non-integral blocks can be computed by the Ext$^1$-quiver of $\D.$

By \cite[Theorem 15.2.4]{Mu12}, every primitive ideal in $U(\mf g)$ is in the set
$$ \{J_\la:=\text{Ann}_{U(\g)}L_\la\;|\,\lambda\in\mf h^\ast\},$$
where $\text{Ann}_{U(\g)}L_\la$ denotes the annihilator of $L_\la$ in $U(\g)$.

We recall the {\em completed Kazhdan-Lusztig order} $\KL^c$ on~$\h^\ast$ as defined in \cite[Section 2.9]{Co16} (also see, \cite[Definition~5.3]{CM16}, as the partial quasi-order transitively generated by $\nu\KL^c\lambda$ when~$L_\la$ is a subquotient of $\Tw_sL_\nu$ for some simple reflection~$s$.  A description of the primitive spectrum  for an arbitrary basic classical Lie superalgebra has been obtained in \cite{Co16} in terms of the ordering $\KL^c$.
\begin{thm}{\rm (}\cite[Theorem 7.2]{Co16}{\rm )} \label{thm::Kevinaminthm}
	For any $\lambda,\nu\in \h^\ast$ we have
	$$J_\nu\subset J_\la\quad\Leftrightarrow\quad \nu\KL^c\lambda.$$
\end{thm}

Theorem \ref{thm::Kevinaminthm} allows us to reduce  the problem of primitive spectrum to the problem of finding composition factors in twisted  simple modules.  For Lie superalgebras  $\gl(m|n)$, it is proved in \cite[Theorem 7.4]{Co16} that it can be computed by the~$\text{Ext}^1$-quiver of $\gl(m|n)$, also see  \cite[Conjecture~5.7]{CM18}.

 In this section, we shall prove a $\D$-analogue and use the reduction method to  answer  the problem of primitive spectrum for non-integral $1$- and $2$-integer $\D$-blocks. 

Let $\la \in \h^\ast$ and $\alpha \in \Phi^+_\oo$.   Then  {$L_\la$} is said to be   $s_\alpha$-free if the $f_\alpha$ acts freely on $L_\la$. Also,  {$L_\la$} is said to be $s_\alpha$-finite   if $L_\la$ is not $s_\alpha$-free. We refer the reader to \cite[Section 4.3.5]{CCC19} for the characterization of $s_\alpha$-freeness for $\D$. We recall that  $\Pi(\g_\oa)$ denotes the simple system of the root system of $\mf g_\oa$.
The following useful property holds for any basic classical Lie superalgebra.

\begin{lem}{\rm (}\cite[Theorem 5.12]{CM16}{\rm )}\label{lem51}
	Let $\mu \in \h^\ast$, $\alpha \in \Pi(\g_\oa)$ and $s:=s_{\alpha}$. Then we have:
	\begin{itemize}
		\item[(i)]  If $L_\mu$ is $s$-finite then $\Tw_sL_\mu =0$.
		\item[(ii)]  If $L_\mu$ is $s$-free and $\langle\mu, \alpha^\vee\rangle \notin \Z$ then $\Tw_sL_\mu$ is irreducible.
		\item[(iii)]  If $L_\mu$ is $s$-free and $\langle\mu, \alpha^\vee\rangle \in \Z$ then we have a short exact sequence
		$$0\rightarrow \emph{Rad} \Tw_sL_\mu \rightarrow \Tw_sL_\mu \rightarrow L_\mu \rightarrow  0.$$
	\end{itemize}	
\end{lem}The radical $U^s_\mu:= \text{Rad} \Tw_sL_\mu $ is referred to as the {\em Jantzen middle}.
\vskip 0.5cm

We extend in this section the definition of $s$-finite to arbitrary module. That is, we call $M$ $s$-finite if all its composition factors are $s$-finite.

\begin{lem}{\rm (}\cite[Proposition 6.2]{Co16}{\rm )} \label{lem::Co16P62}
	If $L_\mu$ is $s$-free, then the Jantzen middle $U^s_\mu$ is isomorphic to the largest $s$-finite quotient of $\emph{Rad}P_\mu$.
\end{lem}

If $\mf g$ is a semisimple Lie algebra, then the semisimplicity of $U^s_\mu$, for all $s$ and for~$\lambda$ regular, is equivalent to the  validity of the Kazhdan-Lusztig  conjecture \cite{KL79} by \cite[Section~7]{AS03}. For the Lie superalgebras $\gl(m|n)$, the semisimplicity of the Jantzen middles has been established in \cite[Theorem 6.10]{Co16} using the validity of  Brundan-Kazhdan-Lusztig conjecture. 

We now prove a $\D$-analogue. We first note that  Jantzen middles in generic atypical blocks are zeros by Lemma \ref{lem51}.  The following lemma will be useful.
\begin{lem} \label{lem::55}
	Suppose $M\in \mc O$ has  multiplicity-free  composition factors. Then $M$ is semisimple if and only if ${\rm Soc}M \cong {\rm Rad}M$.
\end{lem}
\begin{proof}
	It remains to prove that ${\rm Soc}M \cong {\rm Rad}M$ implies the semisimplicity of $M$.
	
	Let $i:L \rightarrow M$ be an embedding of simple submodule $L$. By the assumption, there is an epimorphism $p: M\rightarrow L$. We now claim that $p\circ i: L \rightarrow L$ is an isomorphism. Suppose on the contrary that  $p\circ i=0$ then ${\rm Ker}(p) \supset i(L)$ and so $[M:L]>1$, a contradiction. This means  $p\circ i$ is an isomorphism. Consequently, $i$ is a section of $p$, namely, we have $M =i(L)\oplus {\rm Ker}(p)$.
	
	If ${\rm Ker}(p) \neq 0$ then we consider an  embedding $i': L' \rightarrow {\rm Ker}(p)\subset M$ of a simple submodule $L'$. Using the similar argument we have that $i'(L')$ is a direct summand of ${\rm Ker}(p)$. By repeating the same argument we may conclude that $M$ is semisimple.
\end{proof}

Set $s=s_\alpha$ for some $\alpha \in \Pi(\mf g_\oa)$.   In this case when $\la$ is an  $1$-integer weight, we let $\rm{Ind}_{\mf p}^\g(\cdot): \OO(\gl(2|1),\mf b)_0 \rightarrow \OO'_\la$ denote the  equivalence of parabolic induction with inverse $(\cdot)^{\mf u}$ established in Sections   \ref{Sect::Eqiv0yz}--\ref{Sect::Eqivx0z}. We note that $\alpha$ is a root of $\mf l$ if and only if  $\langle\la, \alpha^\vee\rangle \in \Z$. Therefore, we have $U^s_\mu=0$ if $\alpha$ is not a root of $\mf l$ by Lemma \ref{lem51}.

We now in a position to prove the following theorem.

\begin{thm} \label{thm::ssJantzenmiddle} Let $\la \in \h^\ast$ be a non-integral atypical weight. Then every Jantzen middles in $\mc O_\la$ is either zero or semisimple. In particular, if  $L_\mu\in \mc O_\la$ is $s$-free then we have
	$$U^s_\mu\;\cong\;\bigoplus_{L_\nu:~ \,s  \,{\rm \emph{-} finite}}L_\nu^{\oplus m(\mu,\nu)}$$
	where~$m(\mu,\nu)=\dim\emph{Ext}^1_{\mc O}(L_\mu,L_\nu)$.
	
In addition, in the case when $\la$  is an $1$-integer weight we have the following: If  $\alpha \in \Pi(\mf g_\oa)$ is a root  of $\mf l$ and $s:=s_\alpha$, then  for any  $L_\mu,  L_{\nu} \in \mc O_\la$ such that  $L_\mu$ is $s$-free and $L_\nu$ is $s$-finite  we have $m(\mu,\nu)=\dim\emph{Ext}^1_{\mc O}(L_\mu,L_\nu)= \dim\emph{Ext}^1_{\mc O(\gl(2|1),\mf b)}(L^{\mf a}(\mu),L^{\mf a}(\nu))$.
\end{thm}
\begin{proof}
	If $\la$ is generic then all Jantzen middles are zeros as mentioned above, so we first assume that $\la$ is an $1$-integer weight.
	
	Let   $\alpha$ be a positive root of $\mf l$. Then  for a given  $L_\nu \in \mc O_\la$, we observe that   $L_{\nu} = \rm{Ind}_{\mf p}^\g L^{\mf a}(\nu)$ is $s$-free if and only if ${L^{\mf a}(\nu)} =(L_{\nu})^{\mf u} $ is $s$-free. Therefore  $\rm{Ind}_{\mf p}^\g(\cdot)$ and $(\cdot)^{\mf u}$ provide bijections of $s$-finite modules.  Consequently, $\rm{Ind}_{\mf p}^\g(\cdot)$ gives rise to a bijection between Jantzen middles by Lemma \ref{lem::Co16P62}.  Therefore the semisimplicity of $U^s_\mu$  follows from that of $\gl(2|1)$ by \cite[Theorem 6.10]{Co16}. The desired structure of $U^s_\mu$ now follows from \cite[Theorem 6.8]{Co16}.
	
	We now assume that $\la$ is a $2$-integer weight. By results in Section \ref{Sect::ChFormulae}, all Verma modules in $\mc O_\la$ are multiplicity-free. In particular, each Jantzen middle $U^s_\mu$ is multiplicity-free since $\text{ch}\mathbb T_{s} M_\mu =\text{ch} M_{s  \mu}$ (cf. \cite[Lemma 5.5]{CM16}) and $\mathbb T_s L_\mu$ is an epimorphic image of  $\mathbb T_{s}M_\mu$. By \cite[Corollary 5.14]{CM18} we have   $$\text{Top} U^s_\mu \cong \text{Soc} U^s_\mu,$$ which implies that $U^s_\mu$ is semisimple by Lemma \ref{lem::55}. The proof now follows by \cite[Proposition 6.8]{Co16}.
\end{proof}

As a consequence of Theorem \ref{thm::Kevinaminthm} and Theorem \ref{thm::ssJantzenmiddle}, we have the following corollary, which reduces the problem of primitive spectra for $\D$ to that of $\gl(2|1)$. The latter was computed in \cite{Mu93} (also, see \cite{CM18}).

\begin{cor} Let $\la =(x,y,z)\in \h^\ast$ be a non-integral atypical weight. 
	Then for any  $\mu \in \h^\ast$ we have:
	\begin{itemize}
		\item[(1)] If $\la$ is generic then $J_\la \subset J_\mu$ if and only if $L_\mu = {\mathbb T_{s_\alpha} L_\la},$ for some $\alpha \in \Pi(\mf g_\oa)$. In this case, we have $J_\la = J_\mu$.
		\item[(2)] If  $\la$ is an $1$-integer weight then $J_\la \subset J_\mu$  if and only if either of the following conditions holds:
		\begin{itemize}
			\item[(a)] $L_\mu =\mathbb T_{s_\alpha} L_\la$, for $\alpha \in \Pi(\mf g_\oa)$ with $\langle\la, \alpha^\vee\rangle \notin \Z.$
			\item[(b)] ${{\rm Ann}_{U(\gl(2|1))}}{L^{\mf a}(\la)} \subset {{\rm Ann}_{U(\gl(2|1))}}{L^{\mf a}(\mu)}$.
		\end{itemize}
	\end{itemize}

\end{cor}

\end{document}